\documentclass[10pt]{article}
\usepackage{amsmath,amsthm,amssymb,amsfonts}
\usepackage{enumerate}
\usepackage{mathrsfs}
\usepackage{tikz}

\usepackage[utf8]{inputenc}
\usepackage[expansion=false]{microtype}
\usepackage{amsopn}
\usepackage{eucal}
\usetikzlibrary{diagrams}
\usepackage[colorlinks=false]{hyperref}

\newcommand{\resp}{{\sfcode`\.1000 resp.}}
\newcommand{\ie}{{\sfcode`\.1000 i.e.}}

\newcommand{\eg}{{\sfcode`\.1000 e.g.}}
\newcommand{\cf}{{\sfcode`\.1000 cf.}}

\normalsize

\newtheoremstyle{theoremstyle}
  {10pt}        {5pt}         {\itshape}    {}            {\bfseries}   {:}           {.5em}                      {}          
\newtheoremstyle{examplestyle}
  {10pt}        {5pt}         {}            {}            {\bfseries}   {:}           {.5em}                      {}          

\theoremstyle{theoremstyle}
\newtheorem{theorem}{Theorem}[section]
\newtheorem*{theorem*}{Theorem}
\newtheorem{lemma}[theorem]{Lemma}
\newtheorem{proposition}[theorem]{Proposition}
\newtheorem*{proposition*}{Proposition}
\newtheorem{corollary}[theorem]{Corollary}
\newtheorem*{corollary*}{Corollary}

\theoremstyle{examplestyle}

\newtheorem{definition}[theorem]{Definition}
\newtheorem{definition*}{Definition}
\newtheorem{remark}[theorem]{Remark}
\newtheorem{remark*}{Remark}

\DeclareMathOperator{\Char}{char}

\newcommand{\AAA}{{\mathcal A}}
\newcommand{\MMM}{{\mathcal M}}
\newcommand{\Sch}{\mathrm{Sch}}
\newcommand{\Sm}{\mathrm{Sm}}

\newcommand{\op}{\mathrm{op}}
\newcommand{\HH}{\mathsf{H}}

\newcommand{\MM}{\mathsf{M}}
\newcommand{\K}{K}
\newcommand{\EE}{\mathsf{E}}
\newcommand{\SH}{\mathsf{SH}}
\newcommand{\KGL}{\mathsf{KGL}}

\newcommand{\tr}{\mathrm{tr}}

\newcommand{\fpsl}{\mathrm{fps}\ell'}

\newcommand{\zll}{\bb{Z}_{(\ell)}}
\newcommand{\cdh}{\mathrm{cdh}}
\newcommand{\Nis}{\mathrm{Nis}}
\newcommand{\red}{\mathrm{red}}
\def\DM{\scr D\scr M}
\def\X{\mathcal{X}}
\def\Tr{\mathrm{Tr}}

\newcommand{\G}{\bb{G}}

\newcommand{\A}{\bb{A}}

\newcommand{\ZZ}{{\bb Z}}
\newcommand{\QQ}{{\bb Q}}
\newcommand{\RR}{{\mathsf R}}

\newcommand{\D}{{\mathcal D}}

\let\scr=\mathcal
\let\bb=\mathbf

\def\L{\mathbf L}
\def\R{\mathbf R}

\def\Z{\bb Z}
\def\Q{\bb Q}
\def\C{\bb C}
\def\A{\bb A}
\def\P{\bb P}
\def\1{\mathbf 1}

\def\H{\scr H}
\def\SH{\scr S\scr H}
\def\Sch{\scr S\mathrm{ch}}

\def\Set{\scr S\mathrm{et}}

\def\aast{{\ast\ast}}
\def\et{{\mathrm{\acute et}}}

\def\cdh{{\mathrm{cdh}}}
\def\ldh{{\ell\mathrm{dh}}}

\def\Nis{{\mathrm{Nis}}}

\def\tr{\mathrm{tr}}
\def\pt{\ast}
\def\ph{{\mathord-}}
\def\op{\mathrm{op}}

\def\suchthat{\:\vert\:}
\def\abs #1{\lvert #1\rvert}
\let\into=\hookrightarrow

\let\tens=\otimes
\def\star{{\ast\ast}}
\def\C{\scr C}

\def\Ab{\scr A\mathrm{b}}
\def\Spt{\scr S\mathrm{pt}}

\def\s{s}
\def\Pre{\mathrm P\mathrm{re}{}}

\def\Mod{\operatorname{-mod}}
\def\Sm{\scr S\mathrm{m}}

\def\Cat{\scr C\mathrm{at}}
\def\Cor{\scr C\mathrm{or}}

\def\Sp{\mathrm{Sp}{}}

\DeclareMathOperator{\Hom}{Hom}
\DeclareMathOperator{\Ho}{Ho}

\DeclareMathOperator{\Map}{Map}

\DeclareMathOperator{\Spec}{Spec}

\DeclareMathOperator{\id}{id}

\DeclareMathOperator{\Sq}{Sq}

\let\lim=\relax
\DeclareMathOperator{\lim}{lim}
\DeclareMathOperator{\colim}{colim}
\DeclareMathOperator{\hocolim}{hocolim}
\DeclareMathOperator{\holim}{holim}

\title{{\bf The motivic Steenrod algebra in positive characteristic}}
\author{Marc Hoyois, Shane Kelly, Paul Arne {\O}stv{\ae}r}
\date{\today}
\begin{document}
\maketitle
\begin{abstract}
Let $S$ be an essentially smooth scheme over a field and $\ell\neq\Char S$ a prime number.
We show that the algebra of bistable operations in the mod $\ell$ motivic cohomology of smooth $S$-schemes is generated by the motivic Steenrod operations.
This was previously proved by Voevodsky for $S$ a field of characteristic zero.
We follow Voevodsky's proof but remove its dependence on characteristic zero by using étale cohomology instead of topological realization and by replacing resolution of singularities with a theorem of Gabber on alterations.
\end{abstract}
{\small{\tableofcontents}}
\newpage

\section{Introduction}
\label{section:introduction}

This work is a continuation of \cite{VV:EMspaces}.
The main results of \emph{loc.\ cit.} are the computations, over fields of characteristic zero, of the motives of motivic Eilenberg--Mac Lane spaces and of the algebra of bistable operations in motivic cohomology. Here we generalize these computations to fields of positive characteristic, and further to essentially smooth schemes over fields.
\vspace{0.1in}

Let $k$ be a perfect field and $\ell \neq\Char k$ a prime number. We denote by $\MMM^\star$ the algebra of bistable operations in the motivic cohomology 
of smooth $k$-schemes with coefficients in $\Z/\ell$ introduced in \cite[\S2]{VV:motivicalgebra}. 
Examples of such operations are the reduced power operations $P^i\in\MMM^{2i(\ell-1),i(\ell-1)}$ for $i\geq 1$, the Bockstein operation $\beta\in\MMM^{1,0}$, and the operations given by multiplication by cohomology classes in $\HH^\star(k,\Z/\ell)$; we denote by $\AAA^\star$ the subalgebra of $\MMM^\star$ generated by these operations. The structure of the algebra $\scr A^\star$ was completely determined in \cite{VV:motivicalgebra}.
On the other hand, motivic cohomology with coefficients in $\Z/\ell$ is represented, in the stable motivic homotopy category $\SH(\Sm_k)$, by the motivic Eilenberg--Mac Lane spectrum $\MM\ZZ/\ell$. It follows that any morphism $\MM\ZZ/\ell\to\Sigma^{p,q}\MM\ZZ/\ell$ in $\SH(\Sm_k)$ induces a bistable operation of bidegree $(p,q)$, and in fact that every operation arises in this way. We therefore have the following situation:
\[\scr A^\aast\into\MMM^\star\twoheadleftarrow \MM\ZZ/\ell^\star\MM\ZZ/\ell.\]
Everything just described can be extended to essentially smooth schemes over fields (see Definition~\ref{definition:essentiallySmooth}).
Our main results are gathered in the following theorem.

\begin{theorem}
\label{theorem:main}
Suppose $S$ is a Noetherian scheme of finite Krull dimension that is essentially smooth over a field, 
and let $\ell\neq\Char S$ be a prime number. 
\begin{enumerate}
	\item The operations
		\[\{\beta^{\epsilon_r}P^{i_r}\dotso\beta^{\epsilon_1}P^{i_1}\beta^{\epsilon_0}\suchthat r\geq 0,\; i_j>0,\; \epsilon_j\in\{0,1\},\; i_{j+1}\geq \ell i_j+\epsilon_j\}\]
		form a basis of $\MMM^\star$ as a left $\HH^\star(S,\Z/\ell)$-module. In particular, $\AAA^\star = \MMM^\star$.
	\item The map $\MM\ZZ/\ell^\star\MM\ZZ/\ell\to \MMM^\star$ is an isomorphism.
	\item There is an equivalence of $\MM\ZZ/\ell$-modules \[\MM\ZZ/\ell\wedge\MM\ZZ/\ell\simeq\bigvee_\alpha\Sigma^{p_\alpha,q_\alpha}\MM\ZZ/\ell\] where $(p_\alpha,q_\alpha)$ are the bidegrees of the operations in (1). 
		\end{enumerate}
\end{theorem}

In characteristic zero Theorem~\ref{theorem:main} is an easy consequence of the results of \cite{VV:motivicalgebra} and \cite{VV:EMspaces}.
We emphasize that the identification of the mod $\ell$ motivic cohomology of $\MM\ZZ/\ell$ with $\MMM^\star$ is not trivial and is a new result in positive characteristic.
It implies that the mod $\ell$ motivic cohomology of any motivic spectrum over $S$ has a well-defined left $\AAA^\star$-module structure. 
\vspace{0.1in}

One crucial step in the proof of Theorem \ref{theorem:main} makes use of base change arguments.
It reduces the proof to the case when $S$ is a \emph{perfect} field $k$. 
\vspace{0.1in}

Voevodsky's proof of Theorem~\ref{theorem:main} (1) for fields of characteristic zero (\cite[\S3.4]{VV:EMspaces}) proceeds by realizing the inclusion $\AAA^\star \into \MMM^\star$
as the cohomology of a map of split proper Tate motives $\MMM \to \AAA$ \cite[Lemma 3.50]{VV:EMspaces}.
Since field extensions are conservative for split proper Tate motives \cite[Corollary 2.70]{VV:EMspaces}, this allows him to assume that $k$ is the field of complex numbers.
The calculations concluding the proof use the topological realization functor to compare $\scr M^\aast$ with the classical Steenrod algebra.
The main problems with attempting this proof in positive characteristic are:
\begin{enumerate}
\item[(a)] There is no topological realization functor.
\item[(b)] The only known technique for proving that motives of motivic Eilenberg-Mac Lane spaces are split proper Tate motives uses symmetric products, 
a construction which takes us out of the category $\Sm_k$ of smooth schemes. 
Resolution of singularities is then used to bridge the gap between the motivic homotopy category and its non-smooth analog, 
and likewise for the category of motives.
\end{enumerate}

We solve these problems by replacing the topological realization functor with a cohomological étale realization functor (\S\ref{subsection:etalerealization}), 
and resolution of singularities with a theorem of Gabber on alterations \cite{Gabber1,Gabber2} (see Theorem~\ref{theorem:gabberGlobal}). 
To apply the latter we use the $\ldh$-topology introduced in \cite{KellyThesis}. The $\ldh$-topology is a refinement of the $\cdh$-topology that allows finite flat surjective maps of degree prime to $\ell$ as coverings (see Definition \ref{definition:fpsl}). 
While the $\cdh$-topology is designed to make schemes locally smooth in the presence of resolutions of singularities, Gabber's theorem implies that every separated scheme of finite type over a perfect field admits an $\ldh$-covering by smooth schemes.
\vspace{0.1in}

In \S \ref{section:background}  we first present what is relevant from the theory of presheaves with transfers.  
We proceed by explaining our base change arguments for motivic Eilenberg-Mac Lane spaces and operations in motivic cohomology.
Theorem \ref{theorem:main} is proved in \S \ref{section:mainresults} modulo the commutativity of the ``fundamental square'' formulated in 
Theorem \ref{theorem:CDcomm}, 
whose proof is deferred to \S \ref{section:proofs}.
Our proof of the latter relies on the central result of \cite[\S5]{KellyThesis} which states that $\MM\ZZ_{(\ell)}$-modules satisfy $\ldh$-descent.
For completeness, we repeat in \S\ref{subsection:tfsI}, with minor modifications, the proof of this result from \cite[\S5]{KellyThesis}.
\vspace{0.1in}

In \S \ref{section:applications} we gather some applications of our results and of Gabber's theorem. 
As a direct consequence of Theorem~\ref{theorem:main}, 
we show that for any motivic spectrum $\EE$ over $S$, 
$\MM\Z/\ell_\star \EE$ is a left comodule over the dual motivic Steenrod algebra $\AAA_\star$. 
This fact is essential for adapting Adams spectral sequence techniques to motivic homotopy theory \cite{DI:2010}, \cite{HKO:2011}. 
Secondly, 
we generalize the main result of \cite{RO2} by showing that Voevodsky's category of motives $\DM(\Sm_k,\Z_{(\ell)})$ 
is equivalent to the homotopy category of $\MM\Z_{(\ell)}$-modules when $k$ is perfect and $\ell$ is invertible in $k$.
\vspace{0.1in}

The results of this paper are further applied in \cite{AsokFasel} to study Euler classes and splittings of vector bundles, 
and in \cite{Hoyois} to prove the Hopkins--Morel equivalence relating algebraic cobordism to motivic cohomology as well 
as the computation of the slices of Landweber exact motivic spectra and of the motivic sphere spectrum.
Further applications are being developed in work on Morel's $\pi_1$-conjecture \cite{OrmsbyOestvaer} and on Milnor's 
conjecture on quadratic forms \cite{RO3}.

\vspace{0.1in}
\noindent
{\bf Acknowledgements:} 
It is our pleasure to thank Andrei Suslin for useful suggestions during the preparation of this paper, 
and Chuck Weibel for valuable comments. 
The third author received partial support from the Leiv Eriksson mobility programme and RCN ES479962.  
He would also like to thank the MIT Mathematics Department for its kind hospitality.

\section{Background}
\label{section:background}

\subsection{Main categories and functors}

Let $S$ be a Noetherian scheme of finite Krull dimension. 
We refer to such schemes as \emph{base schemes} for short. 
In this section we set notation by reviewing the homotopy theories of motivic spaces, 
spaces with transfers, 
spectra, 
and spectra with transfers over $S$, 
and the various adjunctions between them. 
We refer to \cite[\S2]{RO2} and \cite[\S1.1]{VV:EMspaces} for more detailed expositions. 

Let $\Sch_S$ be the category of separated schemes of finite type over $S$.
A full subcategory $\scr C$ of $\Sch_S$ is called an \emph{admissible category} if the following hold \cite[\S 0]{VV:EMspaces}:
\begin{enumerate}
\item $S\in\scr C$ and $\A^1_S\in\scr C$.
\item If $X\in\scr C$ and $U\to X$ is \'etale of finite type, then $U\in\scr C$.
\item $\scr C$ is closed under finite products and finite coproducts.
\end{enumerate}
The most important example is the admissible category $\Sm_S$ of separated smooth $S$-schemes of finite type, which we will simply call \emph{smooth schemes}.

If $\scr C$ is an admissible category, let $\Cor(\scr C)$ be the category with the same objects as $\scr C$ but whose morphisms are the finite $S$-correspondences. In the notation of \cite[\S9.1.1]{CD}, the set of morphisms from $X$ to $Y$ in $\Cor(\scr C)$ is the abelian group $c_{0}(X\times_SY/X, \Z)$.\footnote{In \cite{SV:relativecycles} this group is denoted by $c_{\mathit{equi}}(X\times_SY/X,0)$; its definition is however incorrect for nonreduced schemes, see \cite[\S D.3]{CD}.} This is a subgroup of the free abelian group generated by the closed integral subschemes $Z$ of $X\times_SY$ such that the induced morphism $Z \to X$ is finite and dominates an irreducible component of $X$. If $X$ is regular, then $c_{0}(X\times_SY/X, \Z)$ is the entire free abelian group (\cite[Corollary 3.4.6]{SV:relativecycles}) but in general it is a proper subgroup (\cite[Example 3.4.7]{SV:relativecycles}).
The category $\Cor(\scr C)$ is additive, with direct sum given by disjoint union, and it has a symmetric monoidal structure such that the graph functor $\Gamma\colon\scr C\to\Cor(\scr C)$ is symmetric monoidal.

We denote by $\s\Pre^\pt(\scr C)$ the category of pointed simplicial presheaves on $\scr C$. 
If $R$ is any commutative ring, we denote by $\s\Pre^\tr(\scr C,R)$ the category of additive presheaves of simplicial $R$-modules on $\Cor(\scr C)$. 
By left Kan extension, we obtain an adjunction
\begin{equation}
\label{eqn:adjunctions}
R^\tr\colon 
\s\Pre^\pt(\scr C)
\rightleftarrows 
\s\Pre^\tr(\scr C,R)
\colon
u^\tr,
\end{equation}
where the right adjoint is the obvious forgetful functor. 
By the Dold--Kan correspondence, $\s\Pre^\tr(\scr C,R)$ is equivalent to the category of nonnegatively graded (homological) chain complexes of additive 
presheaves of $R$-modules on $\Cor(\scr C)$, but the simplicial description gives the functor $R^\tr$ a symmetric monoidal structure. 
In fact, the adjunction~\eqref{eqn:adjunctions} is a symmetric monoidal Quillen adjunction for the usual projective model structures 
(i.e., the model structures where weak equivalences and fibrations are determined schemewise). 
The associated homotopy categories will be denoted by $\H^\pt(\scr C)$ and $\H^\tr(\scr C,R)$, respectively. 

If $\tau$ is a Grothendieck topology in $\scr C$, consider the classes of maps
\begin{align*}
	W_{\A^1} &=\{(\A^1\times X)_+\to X_+\suchthat X\in\scr C\},\\
	W_{\tau} &=\{\scr X_+\to X_+\suchthat \text{$X\in\scr C$ and $\scr X\to X$ is a $\tau$-hypercover}\}
\end{align*}
in $\s\Pre^\pt(\scr C)$. If $E$ is any set of maps in $\s\Pre^\pt(\scr C)$, we can consider the left Bousfield localization (see \cite{Hirschhorn:2009}) of $\s\Pre^\pt(\scr C)$ at $E$ (\resp{} of $\s\Pre^\tr(\scr C,R)$ at $R^\tr E$). Although the class $W_\tau$ is not essentially small, it is well-known that there exists a set $S\subset W_\tau$ such that an object in $\H^\pt(\scr C)$ is $W_\tau$-local if and only if it is $S$-local. It follows that the left Bousfield localization of $\s\Pre^\pt(\scr C)$ at $W_\tau$ (\resp{} of $\s\Pre^\tr(\scr C,R)$ at $R^\tr W_\tau$) also exists.

The left Bousfield localization of the projective structure on $\s\Pre^\pt(\scr C)$ at $W_{\A^1}\cup W_{\Nis}$ (\resp{} of $\s\Pre^\tr(\scr C,R)$ at $R^\tr(W_{\A^1}\cup W_{\Nis})$) is called the \emph{motivic model structure} and its weak equivalences are the \emph{motivic weak equivalences}. The associated motivic homotopy categories will be denoted by $\H^\pt_{\Nis,\A^1}(\scr C)$ and $\H^\tr_{\Nis,\A^1}(\scr C,R)$, respectively. Occasionally we will consider intermediate localizations or other topologies, in which case we will use self-explanatory notations for the associated homotopy categories, \eg, $\H^\pt_\et(\scr C)$, $\H^\tr_{\cdh,\A^1}(\scr C,R)$, etc.

In general, $\H^\tr_\tau(\scr C,R)$ is \emph{not} equivalent to the unstable derived category of the additive category of 
$\tau$-sheaves of $R$-modules with transfers on $\scr C$. The following lemma provides necessary and sufficient conditions for this to be true.
  
\begin{lemma}\label{lem:compatiblewithtransfers}
	Let $\scr C\subset\Sch_S$ be an admissible category, $R$ a commutative ring, and $\tau$ a topology on $\scr C$. The following assertions are equivalent.
	\begin{enumerate}
		\item $u^\tr R^\tr\colon\s\Pre^\pt(\scr C)\to\s\Pre^\pt(\scr C)$ sends $W_\tau$ to $W_\tau$-local equivalences.
		\item A morphism $f$ in $\s\Pre^\tr(\scr C,R)$ is an $R^\tr W_\tau$-local equivalence if and only if $u^\tr(f)$ is an $W_\tau$-local equivalence.
		\item The square
		\begin{tikzmath}
			\diagram{\H^\tr(\scr C,R) & \H^\tr_\tau(\scr C,R) \\
			\H^\pt(\scr C) & \H^\pt_\tau(\scr C) \\};
			\arrows (11-) edge node[above]{$\L\id$} (-12) (21-) edge node[below]{$\L\id$} (-22) (11) edge node[left]{$\R u^\tr$} (21) (12) edge node[right]{$\R u^\tr$} (22);
		\end{tikzmath}
		is commutative.
	\end{enumerate}
\end{lemma}

\begin{proof}
	The implication $(2)\Rightarrow(1)$ is obvious, and the equivalence $(2)\Leftrightarrow(3)$ is essentially formal.
	Let us prove $(1)\Rightarrow(2)$. Assume (1) and let $f$ be a morphism in $\s\Pre^\tr(\scr C,R)$. Since the functor $u^\tr\colon\s\Pre^\tr(\scr C,R)\to\s\Pre^\pt(\scr C)$ takes values in the subcategory of simplicial radditive functors on $\scr C_+$, \cite[Theorem 4.20]{VV:srf} shows that, if $f$ is an $R^\tr W_\tau$-local equivalence, then $u^\tr(f)$ is an $W_\tau$-local equivalence. Conversely, suppose that $u^\tr(f)$ is an $W_\tau$-local equivalence. Choose a map $\tilde f$ between $R^\tr W_\tau$-local objects in $\s\Pre^\tr(\scr C,R)$ which is $R^\tr W_\tau$-locally equivalent to $f$. Since we already proved that $u^\tr$ preserve equivalences, $u^\tr(\tilde f)$ is $W_\tau$-locally equivalent to $u^\tr(f)$, and since the source and target of $u^\tr(\tilde f)$ are $W_\tau$-local, $u^\tr(\tilde f)$ is a projective equivalence. Therefore, $\tilde f$ is a projective equivalence, and hence $f$ is an $R^\tr W_\tau$-local equivalence as desired.
\end{proof}

\begin{definition}
	Let $\scr C\subset\Sch_S$ be an admissible category and let $R$ be a commutative ring. A topology $\tau$ on $\scr C$ is \emph{compatible with $R$-transfers} if the equivalent conditions of Lemma~\ref{lem:compatiblewithtransfers} are satisfied.
\end{definition}

This definition agrees with \cite[Definition 10.3.2]{CD}. By \cite[Proposition 10.3.3]{CD}, the Nisnevich topology is compatible with transfers on any admissible category $\scr C$, and by \cite[Proposition 10.4.8 and Proposition 10.3.17]{CD}, the $\cdh$-topology is compatible with transfers on $\Sch_S$. Both facts are also proved in \cite[\S1]{VV:EMspaces}.
 
An object of $\H^\pt(\scr C)$ is called \emph{$\A^1$-local} (\resp{} \emph{$\tau$-local}) if it is $W_{\A^1}$-local (\resp{} $W_\tau$-local) in the sense of \cite[Definition 3.1.4 (1) (a)]{Hirschhorn:2009}. Since $S$ is Noetherian and of finite Krull dimension, a simplicial presheaf $F$ is Nisnevich-local if and only if $F(\varnothing)$ is contractible and $F(Q)$ is homotopy cartesian for any cartesian square
\begin{equation*}
	Q=
	\begin{tikzpicture}
		\diagram{W & V \\ U & X \\};
		\arrows (11-) edge (-12) (11) edge (21) (21-) edge node[above]{$i$} (-22) (12) edge node[right]{$p$} (22);
	\end{tikzpicture}
\end{equation*}
in $\scr C$, where $i$ is an open immersion, $p$ is étale, and $p$ induces an isomorphism $Z\times_XV\cong Z$ for some closed complement $Z$ of $i(U)$; we will call such a square a \emph{Nisnevich square}. By the general principles of left Bousfield localization, $\H^\pt_{\Nis,\A^1}(\scr C)$ can be identified with the full subcategory of $\H^\pt(\scr C)$ spanned by the $\A^1$- and Nisnevisch-local objects. Similarly, $\H^\tr_{\Nis,\A^1}(\scr C,R)$ is the full subcategory of $\H^\tr(\scr C,R)$ spanned by the $\A^1$- and Nisnevich-local objects (\ie, those objects $F$ such that $u^\tr F$ is $\A^1$- and Nisnevich-local).

We note that the forgetful functor $u^\tr\colon\s\Pre^\tr(\scr C,R)\to\s\Pre^\pt(\scr C)$ detects motivic weak equivalences. Indeed, $u^\tr$ sends $R^\tr W_{\Nis}$ to equivalences since the Nisnevich topology is compatible with $R$-transfers, and it is an easy exercise to show that it also sends $R^\tr W_{\A^1}$ to $\A^1$-local equivalences (see \cite[Theorem 1.7]{VV:EMspaces}). As a result, we will also denote by $u^\tr$ the induced functor $\H^\tr_{\Nis,\A^1}(\scr C,R)\to\H^\pt_{\Nis,\A^1}(\scr C)$, right adjoint to $\L R^\tr$.

An inclusion of admissible categories $i\colon \scr C\into\scr D$ induces Quillen adjunctions $(i_!,i^\ast)$ on the model categories $\s\Pre^\pt(\ph)$ and $\s\Pre^\tr(\ph,R)$ with any of the model structures considered, where $i^\ast$ is the restriction functor. These adjunctions commute with the adjunctions~\eqref{eqn:adjunctions} since the right adjoints obviously commute. Since the functor $i_!$ is fully faithful, its derived functor $\L i_!$ will also be fully faithful provided that $i^\ast$ preserve weak equivalences. It is clear that $i^\ast$ always preserves $\A^1$-local equivalences. If $\tau$ is a topology on $\scr D$, then $i^\ast\colon\s\Pre^\pt(\scr D)\to\s\Pre^\pt(\scr C)$ preserves $\tau$-local equivalences if and only if $i$ is cocontinuous for $\tau$ \cite[III, Definition 2.1]{SGA4}. For $i^\ast\colon \s\Pre^\tr(\scr D,R)\to \s\Pre^\tr(\scr C,R)$ to preserve $R^\tr W_\tau$-local equivalences, we have the following criterion.

\begin{lemma}\label{lem:cocontinuousTransfers}
	Let $i\colon\scr C\into\scr D$ be an inclusion of admissible subcategories of $\Sch_S$, $\tau$ a topology on $\scr D$, and $R$ a commutative ring. Suppose that $i$ is cocontinuous for $\tau$ and that $\tau$ is compatible with $R$-transfers on $\scr D$. Then $\tau$ is also compatible with $R$-transfers on $\scr C$ and the restriction functor $i^\ast\colon\s\Pre^\tr(\scr D,R)\to\s\Pre^\tr(\scr C,R)$ preserves $R^\tr W_\tau$-local equivalences.
\end{lemma}

\begin{proof}
	For the first claim we must verify that $u^\tr R^\tr\colon\s\Pre^\pt(\scr C)\to\s\Pre^\pt(\scr C)$ sends $\tau$-hypercovers in $\scr C$ to $\tau$-local equivalences. Since the functors $i_!$ are fully faithful, we have
	\[u^\tr R^\tr\cong u^\tr i^\ast i_! R^\tr \cong i^\ast u^\tr R^\tr i_!.\]
	Now $u^\tr R^\tr i_!$ sends $\tau$-hypercovers in $\scr C$ to $\tau$-local equivalences on $\scr D$ by the assumption that $\tau$ is compatible with $R$-transfers on $\scr D$, and since $i$ is cocontinuous, $i^\ast$ preserves $\tau$-local equivalences. The second claim follows from the isomorphism $i^\ast u^\tr\cong u^\tr i^\ast$ and Lemma~\ref{lem:compatiblewithtransfers} (2).
\end{proof}

The Nisnevich topology satisfies the hypotheses of Lemma~\ref{lem:cocontinuousTransfers} for any $i\colon\scr C\into\scr D$, so that the functors $i^\ast$ preserve motivic weak equivalences. In particular, the derived functors $\L i_!\colon\H^\pt_{\Nis,\A^1}(\scr C)\to\H^\pt_{\Nis,\A^1}(\scr D)$ and $\L i_!\colon \H^\tr_{\Nis,\A^1}(\scr C,R)\to\H^\tr_{\Nis,\A^1}(\scr D,R)$ are fully faithful.

We now turn to the stable theory. As $(\G_m, 1)$ is not projectively cofibrant, we choose a projectively cofibrant replacement $\G$ ({\cf} \cite[Section 2.2]{RO2}). If $p\geq q\geq 0$, define as usual \[S^{p,q}=(S^1)^{\wedge p-q}\wedge \G^{\wedge q}\in\s\Pre^\pt(\scr C).\]
According to the general principles of \cite{Hovey:2001}, the motivic model structures induce symmetric monoidal stable model structures on the category \[\Spt(\scr C)=\Sp^\Sigma(\s\Pre^\pt(\scr C),S^{2,1})\] of symmetric $S^{2,1}$-spectra and on the category \[\Spt^\tr(\scr C,R)=\Sp^\Sigma(\s\Pre^\tr(\scr C,R),R^\tr S^{2,1})\] of symmetric $R^\tr S^{2,1}$-spectra. The fibrant objects are the levelwise motivically fibrant spectra $(E_0,E_1,\dotsc)$ such that the maps $E_i\to\Omega^{2,1}E_{i+1}$ adjoint to the bonding maps are motivic weak equivalences (\cite[Definition 14]{RO2}). The associated homotopy categories will be denoted $\SH(\scr C)$ and $\DM(\scr C,R)$, respectively.

There is a commutative diagram of symmetric monoidal Quillen adjunctions
\begin{tikzequation}\label{eqn:transfers}
	\diagram{
	\s\Pre^\pt(\scr C) & \s\Pre^\tr(\scr C,R) \\
	\Spt(\scr C) & \Spt^\tr(\scr C,R) \\
	};
	\arrows
	(11) edge[vshift=\dbl] node[above=\dbl]{$R^\tr$} (12) edge[<-,vshift=-\dbl] node[below=\dbl]{$u^\tr$} (12)
	(21) edge[vshift=\dbl] node[above=\dbl]{$R^\tr$} (22) edge[<-,vshift=-\dbl] node[below=\dbl]{$u^\tr$} (22)
	(11) edge[vshift=-\dbl] node[left=\dbl]{$\Sigma^\infty$} (21) edge[<-,vshift=\dbl] node[right=\dbl]{$\Omega^\infty$} (21)
	(12) edge[vshift=-\dbl] node[left=\dbl]{$\Sigma^\infty$} (22) edge[<-,vshift=\dbl] node[right=\dbl]{$\Omega^\infty$} (22);
\end{tikzequation}
where the functors in the lower row are defined levelwise. We will denote by $\1$ the motivic sphere spectrum $\Sigma^\infty S_+$ in $\Spt(\scr C)$. When there is no danger of confusion, we will sometimes omit $\Sigma^\infty$ from the notations.

We make the following observation which is lacking from \cite[\S2]{RO2}.

\begin{lemma}\label{lem:utrstable}
	The functor $u^\tr\colon\Spt^\tr(\scr C,R)\to\Spt(\scr C)$ detects stable motivic weak equivalences.
\end{lemma}

\begin{proof}
It detects levelwise motivic equivalences since $u^\tr\colon\s\Pre^\tr(\scr C,R)\to\s\Pre^\pt(\C)$ detects motivic equivalences. 
Define a functor $Q\colon\Spt(\scr C)\to\Spt(\scr C)$ by $(QE)_n=\Hom(S^{2,1}, E_{n+1})$ (with action of $\Sigma_n$ induced by that of $\Sigma_{n+1}$), 
and let $Q^\infty E=\colim_{n\to\infty}Q^nE$. 
Similarly, 
let $Q^\tr\colon\Spt^\tr(\scr C,R)\to\Spt^\tr(\scr C,R)$ be given by $(Q^\tr E)_n=\Hom(R^\tr S^{2,1}, E_{n+1})$. 
Then a morphism $f$ in $\Spt(\scr C)$ (\resp{} in $\Spt^\tr(\scr C,R)$) is a stable motivic equivalence if and only if $Q^\infty(f)$ (\resp{} $Q^\infty_\tr(f)$) is a 
levelwise motivic equivalence. 
The proof is completed by noting that $u^\tr Q^\infty_\tr\cong Q^\infty u^\tr$.
\end{proof}

As a result, we simply denote by $u^\tr\colon\DM(\scr C,R)\to\SH(\scr C)$ the induced functor on homotopy categories.

\subsection{Change of base scheme}
\label{sub:basechange}

In what follows we suppose given a class $\scr C$ of morphisms in the category of base schemes which is stable under base change and such that, for any base scheme $S$, the category $\scr C_S$ of $S$-schemes whose structure map is in $\scr C$ forms an admissible category. Denote by $\scr E(\ph)$ any of the model categories $\s\Pre^\pt(\ph)$, $\s\Pre^\tr(\ph,R)$, $\Spt(\ph)$, $\Spt^\tr(\ph,R)$. As $S$ varies, $\scr E(\scr C_S)$ is then a monoidal $\scr C$-fibered model category over the category of base schemes, in the sense of \cite[\S1.3.d]{CD}. In particular, a morphism of base schemes $f\colon T\to S$ induces a symmetric monoidal Quillen adjunction
\[f^\ast\colon \scr E(\scr C_S)\rightleftarrows \scr E(\scr C_T):f_\ast,\]
where $f^\ast$ is induced by the base change functors $f^\ast\colon\scr C_S\to\scr C_T$ and $f^\ast\colon\Cor(\scr C_S)\to\Cor(\scr C_T)$.
If $f$ belongs to $\scr C_S$, there is also a Quillen adjunction
\[f_\sharp\colon \scr E(\scr C_T)\rightleftarrows \scr E(\scr C_S):f^\ast,\]
where $f_\sharp$ is induced by the forgetful functors $f_\sharp\colon\scr C_T\to\scr C_S$ and $f_\sharp\colon\Cor(\scr C_T)\to\Cor(\scr C_S)$.
The $\scr C$-fibered structure of $\scr E(\scr C)$ is moreover compatible with all the standard adjunctions between the various choices of $\scr E$, in the sense that the functors $f^\ast$ and $f_\sharp$ commute with all left adjoints. For the adjunction $(R^\tr,u^\tr)$, this follows from the commutativity of the squares
\[
\begin{tikzpicture}
	\diagram{\scr C_S & \Cor(\scr C_S) \\ \scr C_T & \Cor(\scr C_T) \\};
	\arrows (11-) edge node[above]{$\Gamma$} (-12) (11) edge node[left]{$f^\ast$} (21) (21-) edge node[below]{$\Gamma$} (-22) (12) edge node[right]{$f^\ast$} (22);
\end{tikzpicture}
\qquad\text{and}\qquad
\begin{tikzpicture}
	\diagram{\C_T & \Cor(\scr C_T) \\ \C_S & \Cor(\scr C_S) \\};
	\arrows (11-) edge node[above]{$\Gamma$} (-12) (11) edge node[left]{$f_\sharp$} (21) (21-) edge node[below]{$\Gamma$} (-22) (12) edge node[right]{$f_\sharp$} (22);
\end{tikzpicture}
\]
from \cite[Lemmas 9.3.3 and 9.3.7]{CD}. For the other adjunctions it is obvious.

Define a new class of morphisms $\widehat{\scr C}$ as follows. A map of base schemes $T\to S$ belongs to $\widehat{\scr C}$ if it is the limit of a cofiltered diagram $(T_\alpha)$ in $\scr C_S$ whose transition maps $T_\beta\to T_\alpha$ are affine and dominant (the dominance condition is needed in the proof of Lemma~\ref{lem:scontinuity} (2) below). If $U$ is any $T$-scheme of finite type, then by \cite[Théorème 8.8.2]{EGA4-3} it is the limit of a diagram of schemes of finite type $(U_\alpha)$ over the diagram $(T_\alpha)$. Moreover, if the morphism $U\to T$ is either
\begin{itemize}
	\item separated,
	\item smooth or étale,
	\item an open immersion or a closed immersion,
\end{itemize}
then we can choose each $U_\alpha\to T_\alpha$ to have the same property (this follows from \cite[Proposition 8.10.4]{EGA4-3}, \cite[Proposition 17.7.8]{EGA4-4}, and \cite[Proposition 8.6.3]{EGA4-3}, respectively). We shall assume that this is also true for $\scr C$-morphisms: we require any $U\in\scr C_T$ to be the limit of schemes $U_\alpha$ with $U_\alpha\in\scr C_{T_\alpha}$ (by the above, this holds if $\scr C=\Sm$). It is then clear that $\widehat{\C}$ is closed under composition.

From now on we fix a $\widehat{\scr C}$-morphism of base schemes $f\colon T\to S$, 
the cofiltered limit of $\scr C$-morphisms $f_\alpha\colon T_\alpha\to S$.

\begin{lemma}\label{lem:scontinuity}
	Consider the categories $\s\Pre^\pt$ and $\s\Pre^\tr$ with the \emph{schemewise} model structures. Let $X\in\C_T$ be a cofiltered limit of schemes $X_\alpha$ in $\C_{T_\alpha}$ and let $p\geq q\geq 0$.
	\begin{enumerate}
		\item For any $F\in\s\Pre^\pt(\scr C_S)$, the canonical map \[\hocolim_\alpha \R\Map(\Sigma^{p,q}(X_\alpha)_+,\L f_\alpha^\ast F)\to \R\Map(\Sigma^{p,q}X_+,\L f^\ast F)\] is an equivalence.
		\item For any $F\in\s\Pre^\tr(\scr C_S,R)$, the canonical map \[\hocolim_\alpha \R\Map(\L R^\tr \Sigma^{p,q}(X_\alpha)_+,\L f_\alpha^\ast F)\to \R\Map(\L R^\tr \Sigma^{p,q}X_+,\L f^\ast F)\] is an equivalence.
	\end{enumerate}
\end{lemma}

\begin{proof}
	Note that $\Sigma^{p,q}X_+$ can be obtained from copies of $(\G_{\mathfrak{m}}^{\times n}\times X)_+$ using finite homotopy colimits in a way which is compatible with base change. Since filtered homotopy colimits commute with finite homotopy limits, we can assume that $p=q=0$. Both sides then preserve homotopy colimits in $F$, so we may further assume that $F=Y_+$ (\resp{} that $F=\L R^\tr Y_+$) where $Y\in\C_S$. Then $\L f^\ast F$ is represented by $Y\times_ST$ and the claim follows from \cite[Théorème 8.8.2]{EGA4-3} (\resp{} from \cite[Proposition 9.3.9]{CD}).
\end{proof}

We now make the following observations.
\begin{itemize}
	\item Any trivial line bundle in $\scr C_T$ is the cofiltered limit of trivial line bundles in $\C_{T_\alpha}$.
	\item Any Nisnevich square in $\scr C_T$ is the cofiltered limit of Nisnevich squares in $\C_{T_\alpha}$.
\end{itemize}
The former is obvious. Any Nisnevich square in $\C_T$ is a cofiltered limit of cartesian squares
\begin{tikzmath}
	\diagram{W_\alpha & V_\alpha \\ U_\alpha & X_\alpha \\};
	\arrows (11-) edge (-12) (11) edge (21) (21-) edge node[above]{$i_\alpha$} (-22) (12) edge node[right]{$p_\alpha$} (22);
\end{tikzmath}
in $\C_{T_\alpha}$, where $i_\alpha$ is an open immersion and $p_\alpha$ is étale. Let $Z_\alpha$ be the reduced complement of $i_\alpha(U_\alpha)$ in $X_\alpha$. It remains to shows that $Z_\alpha\times_{X_\alpha}V_\alpha\to Z_\alpha$ is eventually an isomorphism. By \cite[Corollaire 8.8.2.5]{EGA4-3}, it suffices to show that $Z= \lim_\alpha Z_\alpha$ as closed subschemes of $X$. Now $\lim_\alpha Z_\alpha\cong Z_\alpha\times_{X_\alpha} X$ for large $\alpha$, and so $\lim_\alpha Z_\alpha$ is a closed subscheme of $X$ with the same support as $Z$. Moreover, it is reduced by \cite[Proposition 8.7.1]{EGA4-3}, so it coincides with $Z$.

\begin{lemma}\label{lem:A1Nislocal}
	The functors $\L f^\ast\colon\H^\pt(\scr C_S)\to\H^\pt(\scr C_T)$ and $\L f^\ast\colon\H^\tr(\scr C_S,R)\to\H^\tr(\scr C_T,R)$ preserve $\A^1$-local objects and Nisnevich-local objects.
\end{lemma}

\begin{proof}
	If $f$ is in $\C$ this follows from the existence of the Quillen left adjoint $f_\sharp$ to $f^\ast$ and the observation that $f_\sharp$ sends trivial line bundles to trivial line bundles and Nisnevish squares to Nisnevich squares. Thus, each $\L f_\alpha^\ast$ preserves $\A^1$-local objects and Nisnevich-local objects. Since any trivial line bundle (\resp{} Nisnevich square) over $T$ is a cofiltered limit of trivial line bundles (\resp{} Nisnevich squares) over $T_\alpha$, Lemma~\ref{lem:scontinuity} shows that $\L f^\ast$ preserves $\A^1$-local objects and Nisnevich-local objects in general.
\end{proof}

\begin{lemma}\label{lem:continuity}
	Consider the categories $\s\Pre^\pt$ and $\s\Pre^\tr$ with the \emph{motivic} model structures. Let $X\in\C_T$ be a cofiltered limit of schemes $X_\alpha$ in $\C_{T_\alpha}$ and let $p\geq q\geq 0$.
	\begin{enumerate}
		\item For any $F\in\s\Pre^\pt(\scr C_S)$, the canonical map \[\hocolim_\alpha \R\Map(\Sigma^{p,q}(X_\alpha)_+,\L f_\alpha^\ast F)\to \R\Map(\Sigma^{p,q}X_+,\L f^\ast F)\] is an equivalence.
		\item For any $F\in\s\Pre^\tr(\scr C_S,R)$, the canonical map \[\hocolim_\alpha \R\Map(\L R^\tr \Sigma^{p,q}(X_\alpha)_+,\L f_\alpha^\ast F)\to \R\Map(\L R^\tr \Sigma^{p,q}X_+,\L f^\ast F)\] is an equivalence.
	\end{enumerate}
\end{lemma}

\begin{proof}
	Combine Lemmas \ref{lem:scontinuity} and~\ref{lem:A1Nislocal}.
\end{proof}

It is now easy to deduce a stable version of Lemma~\ref{lem:continuity}. Recall that objects in $\SH(\C_S)$ and $\DM(\C_S,R)$ can be modeled by $\Omega$-spectra, \ie, sequences $(E_0,E_1,\dotsc)$ of $\A^1$- and Nisnevich-local objects $E_i$ with motivic weak equivalences $E_i\simeq\Omega^{2,1}E_{i+1}$. If $E\in\SH(\C_S)$ is represented by the $\Omega$-spectrum $(E_0,E_1,\dotsc)$ and $X\in\H^\pt_{\Nis,\A^1}(\C_S)$, we have
\begin{equation}\label{eqn:Omega}
	[\Sigma^{p,q}\Sigma^\infty X,E]\cong[\Sigma^{p+2r,q+r}X,E_r]
\end{equation}
for any $r\geq 0$ such that $p+2r\geq q+r\geq 0$, and similarly if $E\in\DM(\C_S,R)$ and $X\in\H^\tr_{\Nis,\A^1}(\C_S,R)$.

If $f$ is a $\scr C$-morphism, then the existence of the Quillen left adjoint $f_\sharp$ to $f^\ast$ shows that $\L f^\ast$ commutes with the unstable bigraded loop functors $\R\Hom(S^{p,q},\ph)$. An easy application of Lemma~\ref{lem:continuity} then shows that this is still true if $f$ is a $\widehat{\C}$-morphism. 
Thus, we can explicitly describe the base change functors
\[\L f^\ast\colon\SH(\C_S)\to\SH(\C_T)\quad\text{and}\quad \L f^\ast\colon\DM(\C_S,R)\to\DM(\C_T,R)\]
by sending an $\Omega$-spectrum $(E_0,E_1,\dotsc)$ to the $\Omega$-spectrum $(\L f^\ast E_0,\L f^\ast E_1,\dotsc)$. 
From~\eqref{eqn:Omega} and Lemma~\ref{lem:continuity} we obtain the following result.

\begin{lemma}\label{lem:stablecontinuity}
	Let $X\in\C_T$ be a cofiltered limit of schemes $X_\alpha$ in $\C_{T_\alpha}$ and let $p,q\in\Z$.
	\begin{enumerate}
		\item For any $E\in\SH(\C_S)$, $[\Sigma^{p,q}\Sigma^\infty X_+,\L f^\ast E]\cong\colim_\alpha[\Sigma^{p,q}\Sigma^\infty(X_\alpha)_+,\L f_\alpha^\ast E]$.
		\item For any $E\in\DM(\C_S,R)$, $[\L R^\tr\Sigma^{p,q}\Sigma^\infty X_+,\L f^\ast E]\cong\colim_\alpha[\L R^\tr\Sigma^{p,q}\Sigma^\infty(X_\alpha)_+,\L f_\alpha^\ast E]$.
	\end{enumerate}
\end{lemma}

A morphism in $\widehat{\Sm}$ will be called essentially smooth:

\begin{definition} \label{definition:essentiallySmooth}
A morphism of Noetherian schemes of finite Krull dimension $T \to S$ is said to be \emph{essentially smooth} if it is the limit of a cofiltered diagram $(T_\alpha)$ of smooth $S$-schemes whose transition maps $T_\beta\to T_\alpha$ are affine and dominant.
\end{definition}

For example, if $X\in\Sm_S$ and $x\in X$, the local schemes $\Spec\scr O_{X,x}$, $\Spec\scr O_{X,x}^{\mathit{h}}$, and $\Spec\scr O_{X,x}^{\mathit{sh}}$ 
(corresponding to the Zariski, Nisnevich, and \'etale topologies, respectively) are essentially smooth over $S$. The following lemma shows that an essentially smooth scheme over a field is in fact essentially smooth over a finite field $\bb F_p$ or over $\Q$, and in particular over a perfect field. Together with the previous continuity results, this will allow us to extend many results from perfect fields to essentially smooth schemes over arbitrary fields.

\begin{lemma}\label{lem:imperfectext}
	Let $k$ be a perfect field and $L$ a field extension of $k$. Then the map $\Spec L\to\Spec k$ is essentially smooth.
\end{lemma}

\begin{proof}
	Since $\Spec L=\lim_K\Spec K$ where $K$ ranges over the finitely generated field extensions of $k$ contained in $L$, we may assume that $L=k(x_1,\dotsc,x_n)$ for some $x_i\in L$. Since $k$ is perfect, $\Spec k[x_1,\dotsc,x_n]$ has a smooth dense open subset $U$ (\cite[Corollaire 17.15.13]{EGA4-4}). Then $\Spec L$ is the cofiltered limit of the nonempty affine open subschemes of $U$.
\end{proof}

\subsection{Eilenberg--Mac Lane spaces and spectra}
\label{sub:EML}

Let $\scr C\subset\Sch_S$ be an admissible category and let $R$ be a commutative ring. Given $p\geq q\geq 0$ and an $R$-module $A$, the motivic Eilenberg--Mac Lane space $\K(A(q),p)_{\scr C}\in\H_{\Nis,\A^1}^\pt(\scr C)$ is defined by
\[\K(A(q),p)_{\scr C}=u^\tr(\L R^\tr S^{p,q}\tens_R^\L A),\]
\cf{} \cite[\S 3.2]{VV:EMspaces}. Note that this space does not depend on the ring $R$ since $\K(A(q),p)_{\scr C}=u^\tr(\L \Z^\tr S^{p,q}\tens_\Z^\L A)$.
The motivic Eilenberg--Mac Lane spectrum $\MM A_{\scr C}\in\SH(\scr C)$ is defined by
\[\MM A_{\scr C}=u^\tr(\L R^\tr \1\tens_R^\L A).\]
More explicitly, $\MM A_{\scr C}$ is given by the sequence of spaces $\K(A(n),2n)_{\scr C}$ with bonding maps $\Sigma^{2,1}\K(A(n),2n)_{\scr C}\to \K(A(n+1),2n+2)_{\scr C}$.
Note that $\1\in\Spt(\scr C)$ is cofibrant and therefore $\MM R_{\scr C}$ is equivalent to the commutative monoid $u^\tr R^\tr\1$ in the symmetric monoidal model category $\Spt(\scr C)$. According to \cite[Proposition 38]{RO2}, there is a symmetric monoidal model category $\MM R_{\scr C}\Mod$ of $\MM R_{\C}$-modules. In addition, since $\MM A_{\scr C}\simeq u^\tr(R^\tr \1\tens_R\tilde A)$ where $\tilde A\to A$ is a cofibrant replacement of $A$ in $\Spt^\tr(\scr C,R)$, the object $\MM A_{\scr C}$ has a canonical structure of $\MM R_{\scr C}$-module. When $\scr C=\Sm_S$, we will write $\K(A(q),p)$ for $\K(A(q),p)_{\scr C}$ and $\MM A$ for $\MM A_{\C}$.

The symmetric monoidal adjunction $(R^\tr,u^\tr)$ between $\Spt(\scr C)$ and $\Spt^\tr(\scr C,R)$ lifts to a symmetric monoidal adjunction
\[\Phi\colon \MM R_{\scr C}\Mod\rightleftarrows \Spt^\tr(\scr C,R):\Psi\]
such that $\Phi(\MM R_{\scr C}\wedge\ph)=R^\tr$. This is in fact a symmetric monoidal Quillen adjunction (see \cite[\S 2]{RO2}). In \S\ref{subsection:HZmod} we will show that it is a Quillen equivalence when $\scr C=\Sm_S$ and $S$ is the spectrum of a perfect field whose characteristic exponent is invertible in $R$. Here we recall a weaker result which holds over any base. Call an $\MM R$-module \emph{cellular} if it is an iterated homotopy colimit of $\MM R$-modules of the form $\Sigma^{p,q}\MM R$ with $p,q\in\Z$. Similarly, an object in $\DM(\Sm_S,R)$ is cellular if it is an iterated homotopy colimit of objects of the form $R^\tr\Sigma^{p,q}\1$ with $p,q\in\Z$.

\begin{lemma}
\label{lem:HZmod}
The derived adjunction $(\L\Phi,\R\Psi)$ between $\Ho(\MM R\Mod)$ and $\DM(\Sm_S,R)$ 
restricts to an equivalence between the full subcategories of cellular objects.
\end{lemma}

\begin{proof}
	The proof of \cite[Corollary 62]{RO2} works with any ring $R$ instead of $\Z$.
\end{proof}

Let $\scr C$ be a class of morphisms of schemes as in \S\ref{sub:basechange}, and let $f\colon T\to S$ be a morphism of base schemes. For any $p\geq q\geq 0$ and any $R$-module $A$ there is a canonical map
\begin{equation}\label{eqn:EMLunstable}
	\L f^\ast \K(A(q),p)_{\scr C_S}\to \K(A(q),p)_{\scr C_T},
\end{equation}
adjoint to the composition
\begin{multline*}
\L R^\tr \L f^\ast u^\tr (\L R^\tr S^{p,q}_S\tens_R^\L A)\simeq \L f^\ast \L R^\tr u^\tr (\L R^\tr S^{p,q}_S\tens_R^\L A)\\
\to \L f^\ast (\L R^\tr S^{p,q}_S\tens_R^\L A) \simeq \L R^\tr \L f^\ast S^{p,q}_S\tens_R^\L \L f^\ast A\simeq \L R^\tr S^{p,q}_T\tens_R^\L A.
\end{multline*}
Similarly, 
by applying~\eqref{eqn:EMLunstable} levelwise we obtain a canonical map 
\begin{equation}\label{eqn:EMLstable}
	\L f^\ast \MM A_{\scr C_S}\to \MM A_{\scr C_T}.
\end{equation}

\begin{theorem}\label{thm:EMLpb}
	Let $U$ be a base scheme and let $f\colon T\to S$ a morphism in $\widehat{\scr C}_U$. Then \eqref{eqn:EMLunstable} and~\eqref{eqn:EMLstable} are equivalences.
\end{theorem}

\begin{proof}
	It suffices to show that~\eqref{eqn:EMLunstable} is an equivalence.  We may clearly assume $S=U$,  
	so that $f$ itself is a $\widehat{\C}$-morphism. It suffices to show that the canonical map 
	\begin{equation}\label{eqn:futr}
		\L f^\ast u^\tr\to u^\tr \L f^\ast
	\end{equation}
	is an equivalence.
	If $f$ is in $\scr C$, then the functors $f^\ast$ have Quillen left adjoints $f_\sharp$ such that $\L f_\sharp\L R^\tr\simeq\L R^\tr \L f_\sharp$, and so \eqref{eqn:futr} is an equivalence by adjunction. In the general case, let $f$ be the cofiltered limit of $\scr C$-morphisms $f_\alpha\colon T_\alpha\to S$. Let $F\in\H^\tr_{\Nis,\A^1}(\scr C_S,R)$.
	To show that $\L f^\ast u^\tr F\to u^\tr \L f^\ast F$ is an equivalence in $\H_{\Nis,\A^1}^\pt(\scr C_T)$, it suffices to show that for any $X\in\scr C_T$, the induced map
	\[\R\Map(X_+,\L f^\ast u^\tr F)\to\R\Map(X_+,u^\tr \L f^\ast F)\simeq\R\Map(\L R^\tr X_+,\L f^\ast F)\]
	is an equivalence. Write $X$ as a cofiltered limit of schemes $X_\alpha$ in $\scr C_{T_\alpha}$. Then by Lemma~\ref{lem:continuity}, the above map is the homotopy colimit of the maps
	\[\R\Map((X_\alpha)_+,\L f_\alpha^\ast u^\tr F)\to\R\Map(\L R^\tr(X_\alpha)_+,\L f_\alpha^\ast F),\]
	which are equivalences since $f_\alpha$ is in $\scr C$.
\end{proof}

Now we specialize to the case when $\scr C=\Sm_S$ and $S$ is essentially smooth over a field. 
If $k$ is a field and $X$ is a smooth $k$-scheme, 
the motivic cohomology groups $\HH^{p,q}(X,A)$ are defined in \cite[Definition 3.4]{MVW} for any abelian group $A$. 
By \cite[Proposition 3.8]{MVW}, 
these groups do not depend on the choice of the field $k$ on which $X$ is smooth. 
More generally, 
if $X$ is an essentially smooth scheme over a field $k$, 
cofiltered limit of smooth $k$-schemes $X_\alpha$, 
we define
\[
\HH^{p,q}(X,A)
=
\colim_\alpha \HH^{p,q}(X_\alpha,A).
\]
This definition does not depend on the choice of the diagram $(X_\alpha)$ since it is unique as a pro-object in the category of smooth $k$-schemes 
\cite[Corollaire 8.13.2]{EGA4-3}. 
Moreover, 
by \cite[Lemma 3.9]{MVW}, 
it is also independant of the choice of $k$.
Theorem 4.24 in \cite{Hoyois} shows that the motivic Eilenberg-Mac Lane spaces and spectra represent motivic cohomology:

\begin{theorem}\label{thm:motivic}
Assume that $S$ is essentially smooth over a field. 
Let $A$ be an $R$-module and $X\in\Sm_S$. For any $p\geq q\geq 0$ and $r\geq s\geq 0$, there is a natural isomorphism
    \[\HH^{p-r,q-s}(X,A)\cong[\Sigma^{r,s}X_+,\K(A(q),p)]\]
	 and the canonical maps
 	 \begin{equation}\label{eqn:Omegaspectrum}
 		 \K(A(q),p)\to\R\Hom(S^{r,s},\K(A(q+s),p+r))
 	\end{equation}
	are equivalences. For any $p,q\in\Z$, there is a natural isomorphism
    \[\HH^{p,q}(X,A)\cong[\Sigma^\infty X_+,\Sigma^{p,q}\MM A].\]
    \end{theorem}

The following consequence of Theorem \ref{thm:motivic} summarizes the standard vanishing results for motivic cohomology,
{\cf} \cite[Corollary 4.26]{Hoyois}.

\begin{samepage}
\begin{corollary}\label{cor:vanish}
	Assume that $S$ is essentially smooth over a field.
	Let $X\in\Sm_S$ and $p,q\in\Z$ satisfy any of the following conditions:
	\begin{enumerate}
		\item $q<0$,
		\item $p>q+d$,
		\item $p>2q$,
	\end{enumerate}
where $d$ is the least integer such that $X$ can be written as a cofiltered limit of smooth $d$-dimensional schemes over a field. 
Then, for any abelian group $A$, $[\Sigma^\infty X_+, \Sigma^{p,q}\MM A]=0$.
\end{corollary}
\end{samepage}

\subsection{Operations in motivic cohomology}

Let $S$ be essentially smooth over a field. We fix a prime number $\ell\neq\Char S$ and we abbreviate $\K(\Z/\ell(n),2n)$ to $K_n$ and $\MM\Z/\ell$ to $\MM$.
Denote by $\MMM^\star$ the algebra of bistable operations in the motivic cohomology of motivic spaces over $S$: 
an element $\phi\in\MMM^{p,q}$ is a collection of natural transformations
\[\phi_{r,s}\colon\tilde\HH^{r,s}(\ph,\Z/\ell)\to\tilde\HH^{r+p,s+q}(\ph,\Z/\ell),\quad r,s\in\Z,\]
on $\H_{\Nis,\A^1}^\pt(\Sm_S)$ such that, under the bistability isomorphisms, $\phi_{r,s}\Sigma^{2,1}=\phi_{r-2,s-1}$. If $\iota_n\in\tilde\HH^{2n,n}(K_n,\Z/\ell)$ is the tautological class, it is clear that we have an isomorphism
\[\MMM^\star\cong\lim_{n\geq 0}\tilde\HH^{\ast+2n,\ast+n}(K_n,\Z/\ell),\quad\phi\mapsto (\phi_{2n,n}(\iota_n))_n\]
(see \cite[Proposition 2.7]{VV:motivicalgebra}).
As a result, by Theorem~\ref{thm:EMLpb}, if $S$ and $T$ are both essentially smooth over a field, a map $f\colon T\to S$ induces a map of algebras
\begin{equation}\label{eqn:M**pullback}
	f^\ast\colon \MMM^\star_S\to\MMM^\star_T.
\end{equation}

Note that any morphism $\MM\to\Sigma^{p,q}\MM$ in $\SH(\Sm_S)$ induces a bistable operation of bidegree $(p,q)$, which defines a canonical map
\begin{equation}\label{eqn:phantom}
	\MM^\star\MM\to \MMM^\star.
\end{equation}
In particular, the Bockstein map $\MM\to\Sigma^{1,0}\MM$ induces an operation $\beta\in\MMM^{1,0}$, and any cohomology class $\alpha\in\HH^{p,q}(S,\Z/\ell)$ defines an element of $\MMM^{p,q}$.
We can also define the reduced power operations
\[P^i\in\MMM^{2i(\ell-1),i(\ell-1)}\]
as follows. If $S$ is the spectrum of a perfect field, these are defined in \cite[\S9]{VV:motivicalgebra}. It is easy to see by inspecting their definition that if $f\colon\Spec k'\to\Spec k$ is an extension of perfect fields, $f^\ast(P^i)=P^i$. Thus, if $f\colon S\to\Spec k$ is essentially smooth where $k$ is a perfect field, the operation $f^\ast(P^i)\in\MMM^\star_S$ is independent of the choice of $f$.

Let $\AAA^\star\subset\MMM^\star$ be the subalgebra generated by 
\begin{itemize}
	\item the reduced power operations $P^i$ for $i\geq 0$,
	\item the Bockstein $\beta$,
	\item the operations $u\mapsto\alpha u$ for $\alpha\in\HH^\star(S,\Z/\ell)$.
\end{itemize}
Clearly, the map~\eqref{eqn:M**pullback} sends $\AAA^\star_S$ to $\AAA^\star_T$, so that we have a commutative square of algebras:
\begin{tikzequation}\label{eqn:A**pullback}
	\diagram{\AAA^\star_S & \AAA^\star_T \\ \MMM^\star_S & \MMM^\star_T\rlap. \\};
	\arrows (11-) edge node[above]{$f^\ast$} (-12) (11) edge[c->] (21) (21-) edge node[below]{$f^\ast$} (-22) (12) edge[c->] (22);
\end{tikzequation}

\section{Main Results}
\label{section:mainresults}

\subsection{The fundamental square}

The following theorem is our key technical result. 
It will be proved in~\S \ref{section:proofs}.

\begin{theorem}
\label{theorem:CDcomm}
Let $k$ be a perfect field, 
$\ell\neq\Char k$ a prime number, 
$R$ a $\Z_{(\ell)}$-algebra, 
and $i\colon\scr C\into\scr D$ an inclusion of admissible subcategories of $\Sch_k$ with $\scr C\subset\Sm_k$. 
Then the square
\begin{tikzmath}
\diagram{
\H_{\Nis,\A^1}^\pt(\scr D) & \H_{\Nis,\A^1}^\pt(\scr C) \\ 
\H^\tr_{\Nis,\A^1}(\scr D,R) & \H^\tr_{\Nis,\A^1}(\scr C,R) \\};
\arrows (11-) edge node[above]{$ i^\ast$} (-12) (11) edge node[left]{$\L R^\tr$} (21) (21-) edge node[below]{$ i^\ast$} (-22) (12) edge node[right]{$\L R^\tr$} (22);
\end{tikzmath}
is commutative.
\end{theorem}

Commutativity of the fundamental square means the canonical natural transformation
\[\L R^\tr i^\ast\to i^\ast\L R^\tr\]
is an isomorphism. 
If $k$ admits resolution of singularities in the sense of \cite[Definition 3.4]{FV}, this holds for any commutative ring $R$ by \cite[Theorem 1.21]{VV:EMspaces}. 
However, 
resolution of singularities is only known to hold for fields of characteristic zero.

\subsection{Motives of Eilenberg--Mac Lane spaces}
\label{subsection:stm}
Let $F$ be a field. Recall that a split proper Tate motive of weight $\geq n$ in $\H^\tr_{\Nis,\A^1}(\scr C,F)$ is a direct sum of object of the form $\mathbf{L}F^\tr S^{p,q}$ with $p\geq 2q$ and $q\geq n$ \cite[Definition 2.60]{VV:EMspaces}.

\begin{theorem}
\label{theorem:splitpropertatemotive}
Let $S$ be essentially smooth over a field of characteristic exponent $c$. Let $A$ be a finitely generated $\Z[1/c]$-module, $F$ a field of characteristic $\neq c$, and $p\geq 2q\geq 0$. Then $\mathbf{L}F^\tr \K(A(q),p)_{\Sm_S}$ is a split proper Tate motive of weight $\geq q$ in $\H^\tr_{\Nis,\A^1}(\Sm_S,F)$.
\end{theorem} 
\begin{proof}
	We abbreviate $\K(A(q),p)_{\scr C}$ to $\K_{\scr C}$.	By Theorem~\ref{thm:EMLpb} we can assume that $S$ is the spectrum of a perfect field. The theorem is obvious if $p=0$, so assume further that $p>0$. Let $\scr D\subset\Sch_S$ be the admissible subcategory of normal schemes and let $i\colon\Sm_S\into\scr D$ be the inclusion. Note that $i^\ast K_{\scr D}\simeq K_{\Sm_S}$. By Theorem~\ref{theorem:CDcomm}, we have
\begin{equation}\label{eqn:i*K}
	i^\ast \L F^\tr \K_{\scr D}\simeq \L F^\tr i^\ast \K_{\scr D}\simeq \L F^\tr \K_{\Sm_S}.
\end{equation}
By \cite[Corollary 3.28]{VV:EMspaces}, $\L F^\tr \K_{\scr D}$ is split proper Tate of weight $\geq q$. We conclude by noting that the adjunction $(\L i_!,i^\ast)$ restricts to an equivalence between the subcategories of split proper Tate motives of weight $\geq q$ since $\L i_!$ is fully faithful.
\end{proof}

We now fix a prime number $\ell\neq\Char S$ and we abbreviate $\MM\Z/\ell$ to $\MM$. The following corollary is assertion (2) of Theorem~\ref{theorem:main}:

\begin{corollary}
\label{cor:Astablemaps}
The canonical map $\MM^\star\MM\to\MMM^\star$ is an isomorphism.
\end{corollary}

\begin{proof}
    This map fits in the exact sequence
    \[0\to\lim^1 \tilde \HH^{p-1+2n,q+n}(K_n,\Z/\ell)\to \MM^{p,q}\MM\to\lim \tilde \HH^{p+2n,q+n}(K_n,\Z/\ell)\to 0,\]
    and we must show that the $\lim^1$ term vanishes. 
    By Theorem~\ref{theorem:splitpropertatemotive}, $\L\Z/\ell^\tr K_n\simeq\Sigma^{2n,n}M_n$ where $M_n$ is split proper Tate of weight $\geq 0$. All functors should be derived in the following computations. Using the standard adjunctions, we get
    \begin{multline*}
        \tilde \HH^{p-1+2n,q+n}(K_n,\Z/\ell)\cong[\Sigma^\infty K_n,\Sigma^{p-1+2n,q+n}\MM]\cong[\Sigma^\infty\Z/\ell^\tr K_n,\Sigma^{p-1+2n,q+n}\Z/\ell^\tr\1]\\
        \cong[\Sigma^{2n,n}\Sigma^\infty M_n,\Sigma^{p-1+2n,q+n}\Z/\ell^\tr\1]\cong[\Sigma^\infty M_n,\Sigma^{p-1,q}\Z/\ell^\tr\1].
    \end{multline*}
	 To show that $\lim^1 [\Sigma^\infty M_n,\Sigma^{p-1,q}\Z/\ell^\tr\1]=0$, it remains to show that the cofiber sequence
    \[\bigoplus_{n\geq 0}\Sigma^\infty M_n\to\bigoplus_{n\geq 0}\Sigma^\infty M_n\to\hocolim_{n\to\infty}\Sigma^\infty M_n\]
	 splits in $\DM(\Sm_S,\Z/\ell)$. If $S$ is the spectrum of a perfect field, this follows from \cite[Corollary 2.71]{VV:EMspaces}. In general, let $f\colon S\to\Spec k$ be essentially smooth where $k$ is a perfect field. Then by Theorem~\ref{thm:EMLpb}, the above cofiber sequence is the image by $f^\ast$ of the corresponding cofiber sequence over $k$, and hence it splits.
\end{proof}

\begin{corollary}
\label{cor:split}
	$\MM\wedge\MM$ is equivalent to an $\MM$-module of the form $\bigvee_\alpha\Sigma^{p_\alpha,q_\alpha}\MM$.
\end{corollary}

\begin{proof}
By Theorem~\ref{theorem:splitpropertatemotive}, 
$\L \Z/\ell^\tr K_n\simeq\Sigma^{2n,n}M_n$ where $M_n$ is split proper Tate of weight $\geq 0$. 
By \cite[Corollary 2.71]{VV:EMspaces} and Theorem~\ref{thm:EMLpb}, 
$\hocolim_{n\to\infty} M_n$ is again a split proper Tate object of weight $\geq 0$, 
\ie, 
can be written in the form
\[
\hocolim M_n\simeq\bigoplus_\alpha \L \Z/\ell^\tr S^{p_\alpha,q_\alpha}
\]
with $p_\alpha\geq 2q_\alpha\geq 0$.
In the following computations, all functors must be appropriately derived. 
We have the equivalences 
\begin{multline*}\textstyle
\Z/\ell^\tr \MM\simeq\Z/\ell^\tr\colim\Sigma^{-2n,-n}\Sigma^\infty K_n
\simeq \colim\Sigma^{-2n,-n}\Sigma^\infty \Z/\ell^\tr K_n\\
\simeq \colim\Sigma^{-2n,-n}\Sigma^\infty \Sigma^{2n,n} M_n
\simeq \colim\Sigma^\infty M_n
\simeq \Sigma^\infty \colim M_n\\
\simeq \Sigma^\infty\bigoplus_\alpha \Z/\ell^\tr S^{p_\alpha,q_\alpha}
\simeq \Z/\ell^\tr \bigvee_\alpha \Sigma^\infty S^{p_\alpha,q_\alpha}.
\end{multline*}
In particular, 
$\Z/\ell^\tr \MM =\Phi(\MM\wedge \MM)$ is cellular. 
By Lemma~\ref{lem:HZmod}, 
we obtain the equivalences
\[
\MM\wedge \MM
\simeq 
\MM\wedge\bigvee_\alpha \Sigma^\infty S^{p_\alpha,q_\alpha}
\simeq
\bigvee_\alpha\Sigma^{p_\alpha,q_\alpha}\MM.\qedhere
\]
\end{proof}

\subsection{Comparison with étale Steenrod operations}
\label{subsection:etalerealization}

In this paragraph, $k$ is an algebraically closed field. We will denote by $a_\et$ the localization functors
\[a_\et\colon\H^\pt(\Sm_k)\to\H^\pt_\et(\Sm_k)\quad\text{and}\quad a_\et\colon\D(\Sm_k)\to\D_\et(\Sm_k).\]
Here $\D(\Sm_k)$ is the homotopy category of chain complexes of presheaves of abelian groups on $\Sm_k$, and $\D_\et(\Sm_k)$ is the full subcategory spanned by chain complexes satisfying étale hyperdescent or, equivalently, the category obtained from $\D(\Sm_k)$ by inverting maps inducing isomorphisms on étale homology sheaves.

The shifted Suslin--Voevodsky motivic complex $\ZZ(1)[1]\in\D_{\Nis}(\Sm_k)$ of weight one is quasi-isomorphic, 
as a chain complex of Nisnevich sheaves, 
to the presheaf represented by the multiplicative group scheme $\G_{\mathfrak{m}}$ \cite[Theorem 4.1]{MVW}. 
By assuming $1/m\in k$, 
the Kummer short exact sequence of étale sheaves $0\to\mu_m\to\G_{\mathfrak{m}}\stackrel{m}{\to}\G_{\mathfrak{m}}\to 0$ produces a quasi-isomorphism
$a_\et \ZZ/m(1)\simeq \mu_m$, whence
\begin{equation}
\label{eqn:motivictoetale1}
a_\et \ZZ/m(q)[p]
\simeq 
\mu_m^{\otimes q}[p]
\qquad(1/m\in k,\quad p,q\in\ZZ).
\end{equation}
This equivalence induces for every pointed simplicial presheaf $X$ on $\Sm_k$ a canonical map
\begin{equation}
\label{eqn:canonical}
\tilde \HH^{p,q}(X,\ZZ/m)
=
\tilde \HH^{p}_\Nis(X,\ZZ/m(q))
\to 
\tilde \HH^p_\et(X,\mu_m^{\otimes q})
\end{equation}
from motivic to \'etale cohomology. 
Moreover, it is easy to show that~(\ref{eqn:canonical}) is compatible with cup products and the bigraded suspension isomorphisms.
In étale cohomology the latter is the canonical isomorphism
\[
\tilde\HH_\et^{p+2}(\G_{\mathfrak{m}}\wedge\Sigma X,\mu_n^{\otimes q+1})
\cong
\tilde\HH_\et^p(X,\mu_n^{\otimes q}).
\]

\begin{lemma}
\label{lemma:motivic=etale}
Let $X$ be a pointed simplicial presheaf on $\Sm_k$. 
The canonical map (\ref{eqn:canonical}) 
\[
\tilde \HH^{p,q}(X,\ZZ/{\ell})
\to 
\tilde \HH^p_\et(X,\mu_{\ell}^{\otimes q})
\]
is an isomorphism when $p\leq q$.
\end{lemma}
\begin{proof} 
The lemma holds for pointed simplicial smooth schemes by~\cite[Theorem 6.17]{VV:MCodd}. 
A standard argument implies that it automatically holds for all simplicial presheaves. In more details, consider the map 
\begin{equation}\label{eqn:Map}
a_\et
\colon 
\R\Map(\ZZ X,\ZZ/{\ell}(q)[p])
\to
\R\Map(\ZZ X, \mu_{\ell}^{\otimes q}[p])
\end{equation}
between derived mapping spaces, which on $\pi_0$ gives the map of the lemma. If $X=Z_+$ for some $Z\in\Sm_k$ and $K$ is any simplicial set, applying the functor $[K,\mathord-]$ to~\eqref{eqn:Map} yields the instance of (\ref{eqn:canonical}) with the pointed simplicial smooth scheme $(Z\times K)_+$. Hence, \eqref{eqn:Map} is a weak equivalence for such $X$. Since an arbitrary simplicial presheaf is a homotopy colimit of representable presheaves and both sides of~\eqref{eqn:Map} transform homotopy colimits into homotopy limits,
the general case follows.
\end{proof}

Since mod $\ell$ motivic cohomology vanishes in negative weights, we obtain:
\begin{corollary}
\label{corollary:motivic=etale}
Let $X$ be a pointed simplicial presheaf on $\Sm_k$. 
\'Etale sheafification induces an isomorphism
\[
\tilde \HH^{\star}(X,\ZZ/\ell)[\tau^{-1}]
\cong 
\tilde \HH^{\ast}_\et(X,\mu_{\ell}^{\otimes \ast}),
\]
where $\tau\in \HH^{0,1}(\Spec k,\Z/\ell)\cong\mu_{\ell}(k)$ is a primitive $\ell$th root of unity.
\end{corollary}

\begin{remark}
	In \cite{Levine:2000} Levine constructs for $X\in\Sm_k$ an isomorphism $\HH^\star(X,\Z/\ell)[\tau^{-1}]\cong\HH^\ast_\et(X,\mu_\ell^{\tens \ast})$. It is likely that this isomorphism is the same as that of Corollary~\ref{corollary:motivic=etale} for representable presheaves.
\end{remark}

When $p\geq q\geq 0$, 
the chain complex $\ZZ/m(q)[p]\in\D_\Nis(\Sm_k)$ is concentrated in nonnegative degrees and its underlying simplicial presheaf 
(via the Dold--Kan correspondence) is, by definition, the motivic Eilenberg--Mac Lane space $\K(\ZZ/m(q),p)\in\H^\pt_\Nis(\Sm_k)$ which represents the functor $\tilde\HH^{p,q}(\ph,\ZZ/m)$. 
The underlying simplicial presheaf of $\mu_m^{\otimes q}[p]$ is the Eilenberg--Mac Lane object $\K(\mu_m^{\otimes q},p)\in\H^\pt_\et(\Sm_k)$ 
which represents $\tilde\HH_\et^p(\ph,\mu_m^{\otimes q})$. 
In view of the commutativity of the square
\begin{tikzmath}
	\diagram{\D^{\leq 0}_\Nis(\Sm_k) & \D^{\leq 0}_\et(\Sm_k) \\
	\H^\pt_\Nis(\Sm_k) & \H^\pt_\et(\Sm_k)\rlap, \\};
	\arrows (11-) edge node[above]{$a_\et$} (-12) (21-) edge node[below]{$a_\et$} (-22) (11) edge (21) (12) edge (22);
\end{tikzmath}
where the vertical arrows are the forgetful functors, we may restate~\eqref{eqn:motivictoetale1} as an equivalence
\begin{equation}
\label{eqn:motivictoetale2}
a_\et \K(\ZZ/m(q),p)
\simeq 
\K(\mu_m^{\otimes q},p)\qquad(1/m\in k,\quad p\geq q\geq 0)
\end{equation}
in $\H^\pt_\et(\Sm_k)$.
Using the shorthands $K_n=\K(\ZZ/\ell(n),2n)$ and $K_n^\et=\K(\mu_\ell^{\otimes n},2n)$, we can write $a_\et K_n\simeq K_n^\et$.
Combining~\eqref{eqn:motivictoetale2} and Corollary~\ref{corollary:motivic=etale}, we get:
\begin{corollary}\label{corollary:unstable}
	Étale sheafification induces an isomorphism
	\[\tilde \HH^{\star}(K_n,\ZZ/\ell)[\tau^{-1}]\cong \tilde \HH^{\ast}_\et(K_n^\et,\mu_{\ell}^{\otimes \ast}).\]
\end{corollary}

Let $\MMM^\star_\et$ be the algebra of bistable étale cohomology operations on $\H^\pt_\et(\Sm_k)$ with twisted $\mu_\ell$-coefficients. We then have a canonical isomorphism
\[
\MMM^{\star}_\et\cong\lim_{n\geq 0} \tilde \HH^{\ast+2n}_\et(K_n^\et, \mu_\ell^{\otimes\ast+n})
\]
where the transition map
\[\tilde \HH^{\ast+2n+2}_\et(K_{n+1}^\et, \mu_\ell^{\otimes\ast+n+1})\to \tilde \HH^{\ast+2n}_\et(K_n^\et, \mu_\ell^{\otimes\ast+n})\]
is induced by the canonical map $\G_{\mathfrak{m}}\wedge \Sigma K_n^\et\to K_{n+1}^\et$ and the suspension isomorphism
\[\tilde \HH^{\ast+2n+2}_\et(\G_{\mathfrak{m}}\wedge \Sigma K_n^\et, \mu_\ell^{\otimes\ast+n+1})\cong\tilde \HH^{\ast+2n}_\et(K_n^\et, \mu_\ell^{\otimes\ast+n}).\]

We now use for the first time our hypothesis on $k$. Since $k$ is algebraically closed, 
$\mu_\ell$ is a constant sheaf and so $K_n^\et$ is equivalent to the constant simplicial presheaf with value $\K(\mu_\ell(k)^{\otimes n},2n)\cong \K(\Z/\ell,2n)$. Moreover, if $\R\Gamma\colon\H_\et(\Sm_k)\to\Ho(\s\Set)$ is the derived global section functor, with left adjoint $c$, then $\R\Gamma (F)\simeq F(\Spec k)$ since $\Spec k$ has no nontrivial étale hypercovers. Thus, the unit $\id\to\R\Gamma \circ c$ is an isomorphism, \ie, $c$ is fully faithful. It therefore induces an isomorphism
$$
\tilde\HH^\ast(\K(\mu_\ell(k)^{\otimes n},2n),A)\cong\tilde\HH^\ast_\et(K_n^\et, A)
$$
for any abelian group $A$.
In particular, there is a canonical isomorphism \[\chi\colon\AAA^\ast\stackrel{\cong}{\to} \MMM^{\ast,0}_\et\] where $\AAA^\ast$ is the topological Steenrod algebra, and multiplication by $\tau$ induces isomorphisms $\MMM^{\ast,i}_\et\cong\MMM^{\ast,i+1}_\et$. If we define
\[P^i_\et=\tau^{i(\ell-1)}\chi(P^i)\in\MMM^{2i(\ell-1),i(\ell-1)}_\et\]
(note that $\tau^{\ell-1}$ is independent of the choice of $\tau$), we obtain the following presentation of the algebra $\MMM^\star_\et$: it is generated by $\tau^{\pm 1}\in\MMM^{0,\pm 1}_\et$, 
the Bockstein $\beta_\et\in\MMM^{1,0}_\et$, 
and the operations $P^i_\et\in\MMM^{2i(\ell-1),i(\ell-1)}_\et$ for $i\geq 1$; if $\ell$ is odd, 
the relations are the usual topological Adem relations, 
while if $\ell=2$ (in which case $P^i=\Sq^{2i}$) we get the topological Adem relations with additional $\tau$-multiples dictated by the second grading.

Since~\eqref{eqn:canonical} is compatible with the suspension isomorphisms and $a_\et$ is a functor, 
it induces an algebra map 
\begin{equation*}\label{eqn:map1}
\phi\colon\MMM^{\star}
\to 
\MMM^{\star}_\et.
\end{equation*}
On the other hand, 
comparing the motivic and étale Adem relations and sending generators to generators yields a well-defined algebra map
\begin{equation*}
\label{eqn:map2}
\psi\colon\AAA^{\star}
\to 
\MMM^{\star}_\et
\end{equation*}
identifying $\AAA^{\star}$ with the subalgebra $\MMM^{\ast,\geq 0}_\et$.

\begin{lemma}
\label{lem:adem}
The following triangle commutes:
\begin{tikzmath}
\diagram{
\AAA^{\star} & \MMM^{\star} \\  
& \MMM^{\star}_\et\rlap. \\};
\arrows (11-) edge[c->] (-12) (11) edge node[below left]{$\psi$} (22) (12) edge node[right]{$\phi$} (22);
\end{tikzmath}
\end{lemma}

\begin{proof}
It suffices to show that $\phi$ maps $P^i$ to $P^i_\et$ and $\beta$ to $\beta_\et$. Since the identification of $a_\et K_n$ with $K_n^\et$ is compatible with cup products, the motivic Cartan formula \cite[Proposition 9.7]{VV:motivicalgebra} shows that the operations $\chi^{-1}(\tau^{-i(\ell-1)}\phi(P^i))\in\AAA^\ast$ satisfy the axioms 
(1)--(5) of \cite[VI, \S 1]{Steenrod:1962} (if $\ell=2$, it shows that the operations $\chi^{-1}(\tau^{-\lfloor i/2\rfloor}\phi(\Sq^i))$ satisfy the axioms (1)--(5) of \cite[I, \S 1]{Steenrod:1962}). 
Therefore, by \cite[VIII, Theorems 3.9 and 3.10]{Steenrod:1962} and the definition of $P^i_\et\in\MMM^\star_\et$, $\phi(P^i)=P^i_\et$.

It remains to show that $\phi(\beta)=\beta_\et$.
Let $\alpha$ denote the identification~\eqref{eqn:motivictoetale2}. 
The right column in the diagram
\begin{tikzmath}
\diagram{
\Omega a_\et \K(\Z/\ell^2(n),2n+1) & \Omega \K(\mu_{\ell^2}^{\tens n},2n+1) \\
a_\et \K(\Z/\ell(n),2n) & \K(\mu_{\ell}^{\tens n},2n) \\
a_\et \K(\Z/\ell(n),2n+1) & \K(\mu_{\ell}^{\tens n},2n+1) \\
a_\et \K(\Z/\ell^2(n),2n+1) & \K(\mu_{\ell^2}^{\tens n},2n+1)  \\
};
\arrows (11) edge (21) (12) edge (22) (11-) edge node[above]{$\Omega\alpha$} node[below]{$\simeq$} (-12) (21) edge node[left]{$a_\et\beta$} (31) (31) edge (41) (22) edge node[right]{$\beta$} (32) (32) edge (42) (21-) edge[dashed] node[above]{$\gamma$} node[below]{$\simeq$} (-22) (31-) edge node[above]{$\alpha$} node[below]{$\simeq$} (-32) (41-) edge node[above]{$\alpha$} node[below]{$\simeq$} (-42);
\end{tikzmath}
is a fiber sequence, 
and so is the left column because $a_\et$ preserves finite homotopy limits. 
Since the bottom square commutes, 
there exists an equivalence $\gamma$ rendering the diagram commutative \cite[Proposition 6.3.5]{Hovey:1999}. 
To show that $\phi(\beta)=\beta_\et$, 
it suffices to show that $\gamma=\alpha$. 
Taking into account that the canonical equivalences
\[
\Omega a_\et \K(\Z/\ell^2(n),2n+1)
\simeq 
a_\et \K(\Z/\ell^2(n),2n)
\quad\text{and}\quad 
\Omega \K(\mu_{\ell^2}^{\tens n},2n+1)
\simeq 
\K(\mu_{\ell^2}^{\tens n},2n)\]
are compatible with $\alpha$, 
we can identify the top square with:
\begin{tikzmath}
\diagram{
a_\et \K(\Z/\ell^2(n),2n) 
& \K(\mu_{\ell^2}^{\tens n},2n) \\
a_\et \K(\Z/\ell(n),2n)
& \K(\mu_{\ell}^{\tens n},2n)\rlap. 
\\};
\arrows (11) edge node[left]{$\mathrm{mod}\,\ell$} (21) (11-) edge node[above]{$\alpha$} (-12) (12) edge node[right]{$\ell$} (22) (21-) edge node[below]{$\gamma$} (-22);
\end{tikzmath}
Here all objects are Eilenberg--Mac Lane objects of degree $2n$. 
The full subcategory of $\H^\pt_\et(\Sm_k)$ spanned by these objects is equivalent, via $\pi_{2n}$, to the category 
of étale sheaves of abelian groups on $\Sm_k$ \cite[Proposition 7.2.2.12]{HTT}. 
Therefore there exists at most one bottom horizontal map making the diagram commutative.
Since this square commutes with $\alpha$ instead of $\gamma$, we must have $\gamma=\alpha$.
\end{proof}

\subsection{Proof of the main theorem}
\label{sub:proofmain}

In this section we prove the remaining parts of Theorem~\ref{theorem:main}, namely (1) and (3). 
Let $S$ be an essentially smooth scheme of a field and $\ell\neq\Char S$ a prime number.

We first prove (1) when $S=\Spec k$ for a perfect field $k$. We already know from the computation of $\AAA^\star$ in \cite{VV:motivicalgebra} that the claimed basis of $\MMM^\star$ is a basis of $\AAA^\star$, so it will suffice to prove that $\AAA^\star=\MMM^\star$.  Theorem~\ref{theorem:splitpropertatemotive} implies,
cf.~\cite[Lemma 3.50]{VV:EMspaces}, that 
there exist split proper Tate motives $\AAA$ and $\MMM$ whose cohomology agrees with $\AAA^{\star}$ and $\MMM^{\star}$, 
respectively, 
and a map $\MMM\to\AAA$ which is invariant under change of perfect base field and induces the inclusion $\AAA^{\star}\hookrightarrow\MMM^{\star}$. 
Thus, $\AAA^{\star}=\MMM^{\star}$ if and only if $\MMM\to\AAA$ is an equivalence. 
By~\cite[Corollary 2.70 (2)]{VV:EMspaces} we may assume that $k$ is algebraically closed. 
In this case, 
$\MM^{\star}\cong\ZZ/\ell[\tau]$ where $\tau\in \HH^{0,1}(\Spec k,\ZZ/\ell)\cong \mu_\ell(k)$ is a primitive $\ell$th root of unity.

We claim that (1) follows from the following statements:
\begin{enumerate}
	\item[(i)] $\AAA^{\star}/\tau\AAA^{\star}\to \MMM^{\star}/\tau\MMM^{\star}$ is injective.
	\item[(ii)] $\AAA^{\star}[\tau^{-1}]\to\MMM^{\star}[\tau^{-1}]$ is surjective.
\end{enumerate}
Indeed, 
let $x\in\MMM^{\star}$. 
By~(ii), 
$\tau^nx$ belongs to $\AAA^{\star}$ for some $n\geq 0$. 
If $n>0$, 
$\tau^nx$ becomes zero in $\MMM^{\star}/\tau\MMM^{\star}$. 
By~(i) it is zero in $\AAA^{\star}/\tau\AAA^{\star}$. 
Thus, 
there exists $y\in\AAA^{\star}$ such that $\tau^nx=\tau y$. 
Since $\MMM^{\star}$ is the cohomology of a split Tate object, 
it has no $\tau$-torsion. 
Hence, 
$\tau^{n-1}x=y$ and so $\tau^{n-1}x$ belongs to $\AAA^{\star}$. 
Induction on $n$ implies $x\in\AAA^{\star}$.

The injectivity of $\AAA^{\star}/\tau\AAA^{\star}\to \MMM^{\star}/\tau\MMM^{\star}$ is proved in~\cite[Proposition 3.56]{VV:EMspaces}. 
To finish the proof we show that $\AAA^{\star}[\tau^{-1}]\to\MMM^{\star}[\tau^{-1}]$ is surjective.
By Lemma~\ref{lem:adem}, there is a commutative diagram
\begin{tikzmath}
\diagram{
\AAA^{\star}[\tau^{-1}] 
& \MMM^{\star}[\tau^{-1}] 
& \tilde \HH^{\ast + 2n, \ast + n}(K_n,\Z/\ell)[\tau^{-1}] \\  
& \MMM^{\star}_\et & \tilde \HH^{\ast + 2n}_\et(K_n^\et,\mu_\ell(k)^{\tens \ast +n}) 
\\};
\arrows (11-) edge (-12) (11) edge node[below left]{$\psi$} node[above right]{$\simeq$} (22) (12) edge node[right]{$\phi$} (22) (12-) edge (-13) (22-) edge (-23) (13) edge node[right]{$\simeq$} (23);
\end{tikzmath}  
where the rightmost map is an isomorphism by Corollary~\ref{corollary:unstable}.
We are thus reduced to showing that $\MMM^{\star}[\tau^{-1}]\to\MMM^{\star}_\et$ is injective. 
If $x=(x_0,x_1,\dotsc)\in\MMM^{\star}$ maps trivially to $\MMM^{\star}_\et$ then $x_n$ maps trivially to 
$\tilde \HH^{\ast +2n,\ast + n}(K_n,\ZZ/\ell)[\tau^{-1}]$ for all $n$. 
By Theorem~\ref{theorem:splitpropertatemotive}, ${\bf L}{\ZZ/\ell}^\tr K_n$ is a split proper Tate motive and in particular $\tilde \HH^{\ast +2n,\ast + n}(K_n,\ZZ/\ell)$ has no $\tau$-torsion. 
It follows that $x_n=0$ for all $n$, whence $x=0$.
This concludes the proof of assertion (1) of Theorem~\ref{theorem:main} when the base is a perfect field.

We now turn to the proof of assertion (3). By Corollaries~\ref{cor:Astablemaps} and~\ref{cor:split}, we have
\begin{equation}\label{eqn:M**}
	\MMM^\star\cong[\MM,\Sigma^{\star}\MM]\cong [\MM\wedge\MM,\Sigma^\star\MM]_{\MM}\cong [\bigvee_\alpha\Sigma^{p_\alpha,q_\alpha}\MM,\Sigma^\star\MM]_{\MM}\cong\prod_{\alpha}\MM^{\ast -p_\alpha, \ast -q_\alpha}.
\end{equation}
To determine this family of bidegrees $(p_\alpha,q_\alpha)$, we can again assume, by Theorem~\ref{thm:EMLpb}, that $S$ is the spectrum of an algebraically closed field, so that $\MM^\star\cong\Z/\ell[\tau]$. In this case we also know by part (1) that $\MMM^\star\cong\bigoplus_{\gamma}\MM^{\ast -r_\gamma, \ast -s_\gamma}$ where the family of bidegrees $(r_\gamma,s_\gamma)$ is the desired one. In particular, $\MMM^{p,q}$ is finite for every $(p,q)\in\Z$, which shows that the product~\eqref{eqn:M**} is a direct sum. Thus, we find a bigraded isomorphism of free $\Z/\ell[\tau]$-modules
\[\bigoplus_{\alpha}\Sigma^{-p_\alpha,-q_\alpha}\Z/\ell[\tau]\cong\bigoplus_{\gamma}\Sigma^{-r_\gamma,-s_\gamma}\Z/\ell[\tau],\]
(where $\Sigma^{i, j}$ indicates shifting of the bidegree), and it follows that the indexing families $(p_\alpha,q_\alpha)$ and $(r_\gamma,s_\gamma)$ coincide.

Finally, we prove (1) of Theorem~\ref{theorem:main} in general. Choose an essentially smooth morphism $f\colon S\to\Spec k$ where $k$ is a perfect field. By Corollary~\ref{cor:vanish}, for any $(p,q)\in\Z$, there are at most finitely many $\alpha$ such that $\MM_S^{p+p_\alpha,q+q_\alpha}\neq 0$. It follows that the product in~\eqref{eqn:M**} is a direct sum, and hence that the map
\[f^\ast\colon \MMM^\star_k\to \MMM^\star_S\]
induces an isomorphism of left $\MM^\star_S$-modules $\MM^\star_S\tens_{\MM^\star_k}\MMM^\star_k\cong\MMM^\star_S$. From the commutative square~\eqref{eqn:A**pullback}, we obtain $\AAA^\star_S=\MMM^\star_S$. The more precise statement of assertion (1) is automatic since $f^\ast\colon\MMM^\star_k\to\MMM^\star_S$ is an algebra map.

\section{Commutativity of the fundamental square}
\label{section:proofs}

In this section we prove Theorem~\ref{theorem:CDcomm}. 
Throughout, 
the base scheme is a perfect field $k$ of characteristic exponent $c$, 
and $\ell\neq c$ is a fixed prime number.

\subsection{The $\ldh$-topology}
\label{subsection:lprimetopology}
We start with the definition of the $\ldh$-topology introduced in \cite{KellyThesis}.
The basic idea is to enlarge the $\cdh$-topology of Suslin--Voevodsky \cite{SV:singularhomology} by including finite flat surjective maps of degree prime to $\ell$. 

We say that a family of maps
$\{ V_j \to X \}_{j \in J}$ is a \emph{refinement} of another family $\{ U_i \to X \}_{i \in I}$ if for each $j\in J$ there exists an $i_j\in I$ and a factorization 
$V_j \to U_{i_j} \to X$. 

\begin{definition}
\label{definition:fpsl}
Let $\ell\neq \Char k$ be a prime number.
\begin{enumerate}
\item 
An \emph{$\fpsl$-cover} (fini-plat-surjectif-premier-à-$\ell$) is a finite flat
surjective map $f\colon U \to X$ such that $f_\ast\scr O_U$ is a free $\scr O_X$-module of rank prime to $\ell$. 
\item 
An \emph{$\ldh$-cover} is a finite family of maps of finite type 
$\{U_i \to X\}$ which admits a refinement $\{V_j'\to V_j\to X\}$ where $\{V_j\to X\}$ is a $\cdh$-cover and each $V_j'\to V_j$ is an $\fpsl$-cover.
\end{enumerate}
\end{definition}

\begin{definition}
\label{definition:ldhtopology}
The \emph{$\ldh$-topology} on $\Sch_k$ is the topology generated by the $\ldh$-covers. If $\scr C\subset\Sch_k$ is a subcategory, the $\ldh$-topology on $\scr C$ is the topology induced by the $\ldh$-topology on $\Sch_k$.
\end{definition}

\begin{remark}
The $\ldh$-topology is a ``global'' version of the topology of $\ell'$-alterations appearing in \cite{Gabber1} and \cite{Gabber2}.
\end{remark}

\begin{theorem}[Gabber {\cite[Theorem 1.3]{Gabber1}, \cite[IX, Theorem 1.1]{Gabber2}}] 
\label{theorem:gabberGlobal}
Let $X\in\Sch_k$ and let $Z\subset X$ be a nowhere dense closed subset. There exists a map $f\colon Y\to X$ in $\Sch_k$ such that:
\begin{enumerate}
	\item $Y$ is smooth and quasi-projective over $k$.
	\item 
	$f$ is proper, surjective, and sends generic points to generic points.
	\item 
	For each generic point $\xi$ of $X$ there is a unique point $\eta$ of $Y$ over it,
	and $[k(\eta):k(\xi)]$ is finite of degree prime to $\ell$.
	\item $f^{-1}(Z)\subset Y$ is a divisor with strict normal crossings.
\end{enumerate}
\end{theorem}

We shall make use of the following formulation of the Raynaud--Gruson flattening theorem \cite{RG}.
\begin{theorem}
[{\cite[Theorem 2.2.2]{SV:relativecycles}}]
\label{theorem:platification}
Let $p: X \to S$ be a map of Noetherian schemes and $U$ an open subscheme in $S$ 
such that $p$ is flat over $U$. 
Then there exists a closed subscheme $Z$ in $S$ such that $U \cap Z = \varnothing$, 
and the proper transform of $X$ with respect to the blow-up $\operatorname{Bl}_Z S \to S$ with center in $Z$ is flat over $S$.
\end{theorem}

\begin{corollary}
\label{corollary:l'resolution}
For every $X\in\Sch_k$ there exists an $\ldh$-cover $\{U_i \to X\}$ where each $U_i$ is smooth and quasi-projective over $k$.
\end{corollary}
\begin{proof}
The proof uses Noetherian induction. 
Suppose the result holds for all proper closed subschemes of $X$. 
We may assume that $X$ is integral since the set of inclusions of irreducible components is a $\cdh$-cover. 
Let $Y \to X$ be the map provided by Theorem~\ref{theorem:gabberGlobal}, so that $Y$ is smooth and quasi-projective over $k$.
By Theorem~\ref{theorem:platification}, 
there exists a blowup with nowhere dense center $X' \to X$ such that the proper transform $Y' \to X'$ of $Y \to X$ is Zariski-locally an $\fpsl$-cover. 
If $Z \subset X$ is a proper closed subscheme such that $X' \to X$ is an isomorphism outside of $Z$, then $\{Z \to X, X' \to X \}$ is a $\cdh$-cover. 
By the inductive hypothesis, 
there exists an $\ldh$-cover $\{Z_j \to Z\}_{j \in J}$ with each $Z_j$ quasi-projective and smooth over $k$. We claim that $\{Z_j\to X\}_{j\in J}\cup\{Y\to X\}$ is an $\ldh$-cover. Indeed, it is refined by $\{Z_j\to X\}_{j\in J}\cup\{Y'\to X\}$ which is the composition of the $\cdh$-cover $\{Z\to X, X'\to X\}$ with the $\ldh$-covers $\{Z_j\to Z\}_{j\in J}$ and $\{Y'\to X'\}$.
\end{proof}

It is easy to show that any presheaf of $\zll$-modules with transfers on $\Sch_k$ is an $\fpsl$-sheaf. 
In particular, 
such a presheaf is a $\cdh$-sheaf if and only if it is an $\ldh$-sheaf. 
More generally, we have:

\begin{theorem}[{\cite[Theorem 3.4.17]{KellyThesis}}]
\label{thm:cdh=ldh}
Suppose $F$ is a presheaf of $\zll$-modules with transfers on $\Sch_k$.
Then for every $X\in\Sch_k$, 
the canonical map
\[ 
H^n_{\cdh}(X, a_\cdh F) 
\to 
H^n_{\ldh}(X, a_{\ldh} F) 
\]
is an isomorphism for all $n\geq 0$.
\end{theorem}

\begin{lemma}
\label{lemma:ThePostnikovtowerargument}
Let $\scr C$ be a small category, $\sigma\leq \tau$ Grothendieck topologies on $\scr C$, 
and $F$ a presheaf of spectra on $\scr C$. 
Assume that:
\begin{enumerate}
\item $F$ is $\sigma$-local.
\item every $X\in\scr C$ has finite $\sigma$-cohomological dimension.
\item for every $X\in\scr C$ and every $p,q\in\Z$, 
the canonical map $H^p_\sigma(X,a_{\sigma}\pi_q F)\to H^p_\tau(X,a_\tau\pi_q F)$ is an isomorphism.
\end{enumerate}
Then $F$ is $\tau$-local.
\end{lemma}

\begin{proof}
Denote by $\Sp(\scr C)$ the triangulated category of presheaves of spectra on $\scr C$. It is endowed 
with a $t$-structure for which the truncation functors $\tau_{\leq n}$ and $\tau_{\geq n}$ are defined objectwise. 
For a topology $\rho$ on $\scr C$, let $i_\rho\colon\Sp_\rho(\scr C)\into\Sp(\scr C)$ be the inclusion of the triangulated subcategory of $\rho$-local presheaves and $a_\rho\colon\Sp(\scr C)\to\Sp_\rho(\scr C)$ its left adjoint. 
The subcategory $\Sp_\rho(\scr C)$ inherits a $t$-structure whose truncation functors are $a_\rho\tau_{\leq n}i_\rho$ and $a_\rho\tau_{\geq n}i_\rho$. Moreover, the functors $a_\rho$ and $i_\rho$ are both left $t$-exact (see \cite[Remark 1.9]{DAGVII}).

By (1), we have $F\in \Sp_\sigma(\scr C)$. We first prove the lemma under the assumption that $F$ is bounded below for the $t$-structure on $\Sp_\sigma(\scr C)$, \ie, that $F\simeq a_\sigma\tau_{\geq k}F$ for some $k\in\Z$.
By (2), the $t$-structure on $\Sp_\sigma(\scr C)$ is left complete (see \cite[\S6.1]{Jardine}), 
so that $F\simeq\holim_{n\to\infty} a_\sigma\tau_{\leq n}F$. 
Since the inclusion $\Sp_\tau(\scr C)\hookrightarrow\Sp(\scr C)$ preserves homotopy limits, 
it suffices to show that $a_\sigma\tau_{\leq n}F$ is $\tau$-local for all $n\in\Z$.
Using the fiber sequences
\[
(\pi_nF)[n]
\to 
\tau_{\leq n} F
\to 
\tau_{\leq n-1}F
\]
and the fact that $a_\sigma\tau_{\leq n}F=0$ for $n<k$, 
we are reduced to proving that $a_\sigma((\pi_nF)[n])$ is $\tau$-local for all $n\geq k$, 
or equivalently that the canonical map $a_\sigma((\pi_nF)[n])\to a_\tau((\pi_nF)[n])$ induces an isomorphism on presheaves of homotopy groups. 
By definition of cohomology, we have
\[
\pi_\ast a_\sigma((\pi_nF)[n])
=
H^{n-\ast}_\sigma(\ph,a_\sigma\pi_n F),
\]
so the desired result holds by (3).

For a general $F$, the previous proof applies to the presheaves of spectra $a_\sigma\tau_{\geq k}F$ for all $k\in\Z$ and shows that
\[a_\sigma\tau_{\geq k}F\simeq a_\tau\tau_{\geq k}F.\]
To complete the proof we observe that, for \emph{any} topology $\rho$ on $\scr C$, the canonical map
\[\colim_{k\to-\infty} i_\rho a_\rho\tau_{\geq k}F\to i_\rho a_\rho F\]
is an equivalence of presheaves of spectra. Its cofiber is $\colim_{k\to-\infty} i_\rho a_\rho\tau_{< k}F$. Since both $i_\rho$ and $a_\rho$ are left $t$-exact, $i_\rho a_\rho\tau_{< k}F$ is a presheaf of $(k-1)$-truncated spectra, and hence $\colim_{k\to-\infty} i_\rho a_\rho\tau_{< k}F$ is a presheaf of spectra which are $(k-1)$-truncated for all $k\in\Z$, \ie, it is contractible.
\end{proof}

\begin{remark}
	We sketch a different proof of Lemma~\ref{lemma:ThePostnikovtowerargument} assuming that the topologies $\sigma$ and $\tau$ have enough points, which holds for all the topologies we consider. The cosimplicial Godement resolutions give rise for every $X\in\scr C$ to spectral sequences $\{_\sigma E_r^{\aast}\}_{r\geq 1}$ and $\{_\tau E_r^{\aast}\}_{r\geq 1}$ with
	\[_\sigma E_2^{p,q}\cong H^p_\sigma(X,a_\sigma\pi_q F)\quad\text{and}\quad _\tau E_2^{p,q}\cong H^p_\tau(X,a_\tau\pi_q F),\]
	and to a morphism of spectral sequences from the former to the latter, which by (3) is an isomorphism starting from the second page. By (2), $_\sigma E_2^{p,\ast}=0$ for $p\gg 0$, so both spectral sequences stabilize after finitely many steps. It follows from the second part of \cite[Proposition 5.47]{Thomason:1985} that $_\rho E_\infty^{\aast}$ is the associated graded of a complete, Hausdorff, and (trivially) exhaustive filtration on $\pi_\ast (a_\rho F)(X)$, for $\rho=\sigma,\tau$. This shows that the canonical map $\pi_\ast (a_\sigma F)(X)\to \pi_\ast (a_\tau F)(X)$ is an isomorphism for all $X\in\scr C$, and hence that $a_\sigma F\simeq a_\tau F$.
\end{remark}

\begin{proposition}\label{prop:ldhTransfers}
	Let $R$ be a $\zll$-algebra.
	\begin{enumerate}
		\item The $\ldh$-topology on $\Sch_k$ is compatible with $R$-transfers.
		\item The $\ldh$-topology on $\Sm_k$ is compatible with $R$-transfers.
		\item The functor $i^\ast\colon\s\Pre^\tr(\Sch_k,R)\to\s\Pre^\tr(\Sm_k,R)$ preserves $R^\tr W_\ldh$-local equivalences.
	\end{enumerate}
\end{proposition}

\begin{proof}
	To prove (1) we must show that, for any $\ldh$-hypercover $f\in W_{\ldh}$, $u^\tr R^\tr(f)$ is an $\ldh$-local equivalence; we will show that it is in fact a $\cdh$-local equivalence. Since $\cdh$ is compatible with transfers on $\Sch_k$ \cite[Lemma 1.24]{VV:EMspaces}, we know from Lemma~\ref{lem:compatiblewithtransfers} that $u^\tr R^\tr(f)$ is a $\cdh$-local equivalence if (and only if) $R^\tr(f)$ is an $R^\tr W_\cdh$-local equivalence, and by assumption it is an $R^\tr W_\ldh$-equivalence. By adjunction, it will suffice to show that if $F\in\s\Pre^\tr(\Sch_k,R)$ is $R^\tr W_\cdh$-local, then $F$ is in fact $R^\tr W_\ldh$-local or, equivalently, $u^\tr (F)$ is $\ldh$-local.

We have a canonical identification of presheaves $\pi_q u^\tr (F )
\cong u^\tr \pi_q (F )$ and since the latter have a
structure of presheaves with transfers, it follows from Theorem~\ref{thm:cdh=ldh} that
$$
H^p_{\cdh}(X, a_{\cdh} \pi_q u^\tr (F)) \cong H_{\ldh}^p (X,
a_{\ldh} \pi_q u^\tr (F )).
$$
Since the $\cdh$-topology is cohomologically bounded \cite{SV:2000}, Lemma \ref{lemma:ThePostnikovtowerargument}
shows that $u^\tr (F)$ is $\ldh$-local (to apply the lemma, we view simplicial abelian groups as connective spectra). This proves (1).
	
	By Corollary~\ref{corollary:l'resolution}, $i\colon\Sm_k\into \Sch_k$ is cocontinuous for the $\ldh$-topology. Assertions (2) and (3) now follow from (1) by virtue of Lemma~\ref{lem:cocontinuousTransfers}.
\end{proof}

\subsection{Descent for $\MM\ZZ_{(\ell)}$-modules}
\label{subsection:tfsI}

In this section we review the proof from \cite[\S5]{KellyThesis} that $\MM\ZZ_{(\ell)}$-modules are $\ldh$-local (see Definition~\ref{def:localSH}), which is a key ingredient in the proof of Theorem~\ref{theorem:CDcomm}.
The argument can be summarized as follows. We know from Ayoub's proper base change theorem that motivic spectra are $\cdh$-local. The notion of ``structure of traces'' (Definition \ref{def:traces}) is designed to bridge the gap between $\cdh$-descent and $\ldh$-descent. A first step is therefore to show that the motivic cohomology spectrum $\MM\Z_{(\ell)}\in\SH(\Sm_k)$ has a structure of traces. Roughly speaking, this follows from the facts that homotopy algebraic $K$-theory $\KGL$ has a structure of traces and that $\MM\Z_{(\ell)}$ is the zeroth slice of $\KGL_{(\ell)}$.

We first recall the notion of a structure of traces introduced in \cite{KellyThesis}.

\begin{definition}\label{def:traces}
Let $\scr E\colon\Sch_k\to\Cat$ be a $2$-functor (or equivalently, a cofibered category over $\Sch_k$) with values in additive categories 
and $F$ an additive section of $\scr E$ (\ie, a section of the associated cofibered category such that $F_{X\amalg Y}\cong(i_1)_\ast F_X\oplus (i_2)_\ast F_Y$).
A \emph{structure of traces} on $F$ is a family of maps $\Tr_f\colon f_\ast F_Y\to F_X$, defined for every finite flat surjective map $f\colon Y\to X$ in $\Sch_k$, 
subject to the following coherence conditions (where $f$ and $g$ are finite flat surjective maps).
	\begin{itemize}
	\item (Additivity) $\Tr_{f\amalg g}=(i_1)_\ast\Tr_f\oplus(i_2)_\ast\Tr_g$.
	\item
	(Functoriality) $\Tr_{\id_X}=\id_{F_X}$, and if $g\colon Z\to Y$ and $f\colon Y\to Z$,
	$\Tr_{fg}=\Tr_{f}\circ f_\ast\Tr_{g}$. 
	\item
	(Base change)
	For every cartesian square
	\begin{tikzmath}
	\diagram{
	W & Z  \\
	Y & X \\};
	\arrows (11) edge node[left]{$q$} (21) (11-) edge node[above]{$g$} (-12)
	(21-) edge node[above]{$f$} (-22) (12) edge node[right]{$p$} (22);
	\end{tikzmath}
	in $\Sch_k$,
	we have $F_p\circ\Tr_{f}=p_\ast\Tr_{g}\circ f_\ast F_q$.
	\item
	(Degree)
	If $f$ is globally free of degree $d$ (\ie, $f_\ast\scr O_Y\cong \scr O_X^d$),
	then $\Tr_{f}\circ F_f=d\cdot\id_{F_X}$.
	\end{itemize}
\end{definition}

When $\scr E$ is the constant functor with value $\Ab$, 
this definition specializes to the notion of \emph{presheaf with traces} on $\Sch_k$. 
The other example that we will use is the $2$-functor $X\mapsto\SH(\Sm_X)$ on $\Sch_k$, which takes values in triangulated categories. 
This is a stable homotopy $2$-functor in the sense of \cite{Ayo07}. In this context, we will simply write $f_\ast$, $f^\ast$, and $f_\sharp$ for 
the \emph{derived} functors $\R f_\ast$, $\L f^\ast$, $\L f_\sharp$ associated to a change of base scheme $f$ (see \S\ref{sub:basechange}).

Suppose that the cofibered category $\scr E$ is bifibered, \ie, that for every $f\colon Y\to X$ in $\Sch_k$ the functor $f_\ast\colon \scr E(Y)\to \scr E(X)$ 
has a left adjoint $f^\ast$. Then, to any object $\EE\in\scr E(\Spec k)$, we can associate a section $X\mapsto\EE_X$ of $\scr E$ as follows: 
if $a\colon X\to\Spec k$ is the structure map of $X\in\Sch_k$, then $\EE_X=a^\ast\EE$, and given $f\colon Y\to X$, the map $\EE_X\to f_\ast \EE_Y$ is 
the unit of the adjunction $(f^\ast,f_\ast)$. Under some mild assumptions on $\scr E$ (for example, axioms 1 to 3 in \cite[\S1.4.1]{Ayo07}), 
this is an additive section of $\scr E$. 
A \emph{structure of traces} on $\EE\in\scr E(\Spec k)$ is a structure of traces on the corresponding section $X\mapsto\EE_X$.

\begin{remark}
	When $\EE=\MM A\in\SH(\Sm_k)$ there is a conflict of notations in that, if $X$ is not smooth over $k$, the object $\EE_X$ is not known to coincide 
	with the motivic Eilenberg--Mac Lane spectrum $\MM A_X$ defined in~\S\ref{sub:EML}. When we speak of a structure of traces on $\MM A$, 
	we mean a structure of traces on the cartesian section $(a\colon X\to\Spec k)\mapsto a^\ast\MM A_k$, \emph{not} on the section $X\mapsto \MM A_X$.
\end{remark}

\begin{remark}\label{remark:traces}
	If $\phi\colon\scr E_1\to \scr E_2$ is a $2$-natural transformation and $F$ is an additive section of $\scr E_1$, 
	it is clear that a structure of traces on $F$ induces a structure of traces on $\phi F$. 
	For example, if $F$ is a section of $X\mapsto\SH(\Sm_X)$ with a structure of traces, 
	then for any $\EE\in\SH(\Spec k)$, $X\mapsto[\EE_X,F_X]$ is a presheaf with traces on $\Sch_k$.
\end{remark}

We recall two important results on traces from \cite{KellyThesis}. 
The first one relates traces to transfers:

\begin{theorem}[{\cite[Theorem 3.8.1 (3)]{KellyThesis}}]
\label{theorem:cdhlprimeComparison}
Suppose $F$ is a presheaf of $\zll$-modules with traces on $\Sch_k$ such that:
\begin{enumerate}
 \item $F(X) \to F(X_{\red})$ is an isomorphism for every $X\in\Sch_k$.
 \item $a_{\Nis}F|_{\Sm_k}$ is an unramified sheaf in the sense of \cite[Definition 2.1]{Morel:2012}.
\end{enumerate}
Then $a_{\cdh}F$ has a canonical structure of presheaf with transfers.
\end{theorem}

The second result on traces from \cite{KellyThesis} concerns the compatibility of traces with the slice filtration. To state it we need a few preliminary definitions.

\begin{definition}
A \emph{weak structure of smooth traces} on $\EE\in\SH(\Sm_k)$ is a family of maps $\Tr_f\colon f_\ast \EE_Y\to \EE_X$, 
defined for every finite, flat, globally free, and surjective map $f\colon Y\to X$ in $\Sm_k$, such that, if $f$ is of degree $d$, 
$\Tr_f\circ \EE_f=d\cdot\id_{\EE_X}$.
\end{definition}

Clearly, a structure of traces on $\EE$ induces a weak structure of smooth traces on $\EE$.

\begin{definition}
Let $R\subset\QQ$ be a subring of the rational numbers.
We say that a motivic spectrum $\EE\in\SH(\Sm_k)$ is \emph{$R$-local} if the abelian group $[\EE,\EE]$ is an $R$-module.
\end{definition}

\begin{theorem}[{\cite[Proposition 4.3.7]{KellyThesis}}]
\label{thm:slicetraces}
Suppose $\EE\in \SH(\Sm_k)$ is $\ZZ[1/c]$-local with a structure of traces, and suppose that $s_q\EE$ has a weak structure of smooth traces for every $q\in\Z$.
Then $s_{q}\EE$ has a structure of traces for every $q\in\ZZ$.
\end{theorem}

\begin{proposition}[{\cite[Proposition 5.2.3]{KellyThesis}}]
	\label{prop:KGLtraces}
	The homotopy algebraic $K$-theory spectrum $\KGL\in\SH(\Sm_k)$ has a structure of traces.
\end{proposition}

\begin{proof}
	Let $f\colon Y\to X$ be a finite flat surjective map in $\Sch_k$. It induces an exact functor $f_\ast$ between the biWaldhausen categories of perfect complexes, whence a morphism of $K$-theory spectra $\Tr_f\colon K(Y)\to K(X)$ in the sense of Thomason--Trobaugh. It follows from standard properties of algebraic $K$-theory that this defines a structure of traces on the presheaf of spectra $X\mapsto K(X)$: for additivity use \cite[1.7.2]{TT90}, for functoriality \cite[1.5.4]{TT90}, for base-change \cite[3.18]{TT90}, and for degree \cite[1.7.3]{TT90}. Given $X\in\Sch_k$, let $K_X\in\Ho(\Spt^{\Sm_X^\op})$ be the restriction of $K$ to $\Sm_X$. The trace maps $\Tr_{f\times_XU}\colon K(Y\times_XU)\to K(U)$ define morphisms of presheaves $\Tr_f\colon f_\ast K_Y\to K_X$ which endow $X\mapsto K_X$ with a structure of traces.
	
	The Bass $K$-theory presheaf $K^B_X$ can be defined by
	\[K^B_X=K_X^\sharp=\hocolim_n\R\Hom(\Sigma^\infty(\P^1)^{\wedge n}, K_X),\]
(see \cite[end of paragraph 2.5, Proposition 2.7, Proposition 2.10]{Cisinski:2013}). For any morphism $f\colon Y\to X$ in $\Sch_k$, the derived pushforward functor $f_\ast\colon\Ho(\Spt^{\Sm_Y^\op})\to\Ho(\Spt^{\Sm_X^\op})$ right adjoint to $f^\ast$ commutes with homotopy colimits by abstract nonsense, and so $f_\ast K^B_X\simeq (f_\ast K_X)^\sharp$. It follows that the structure of traces on the section $X\mapsto K_X$ induces a structure of traces on the section $X\mapsto K^B_X$.
	
	The homotopy algebraic $K$-theory presheaf of spectra $KH_X\colon\Sm_X\to\Spt$ is by definition the $\A^1$-localization of the presheaf $K^B_X$, that is,
	\[KH_X = L_{\A^1}K^B_X = \hocolim_n \R\Hom(\Sigma^\infty(\Delta^n_X)_+,K^B_X),\]
where $\Delta^\bullet_X$ is the usual cosimplicial diagram in $\Sm_X$ with $\Delta^n_X=\A^n_X$. As before, we have $f_\ast KH_Y\simeq L_{\A^1} f_\ast K^B_Y$ and hence an induced structure of traces on the section $X\mapsto KH_X$.

Denote by $\scr E(X)$ the homotopy category of $\P^1$-spectra in the category of presheaves of spectra on $\Sm_X$; its objects are sequences $(E_0,E_1,\dotsc)$ of presheaves of spectra on $\Sm_X$ together with equivalences $E_i\simeq\R\Hom(\Sigma^\infty\P^1,E_{i+1})$. The $2$-functor $X\mapsto\SH(\Sm_X)$ can then be identified with the sub-$2$-functor of $\scr E$ spanned by the $\A^1$- and Nisnevich-local objects. Under this identification, the motivic spectrum $\KGL_X\in\SH(\Sm_X)$ is the $\P^1$-spectrum $(KH_X,KH_X,\dotsc)$ where the equivalences $KH_X\simeq\R\Hom(\Sigma^\infty\P^1,KH_X)$ are given by Bott periodicity. Thus, it remains to observe that the trace maps $\Tr_f\colon f_\ast KH_Y\to KH_X$ are compatible with these equivalences, which follows easily from the definitions.
\end{proof}

\begin{proposition}
\label{prop:MZweakTraces}
For any $\EE\in\DM(\Sm_k,R)$,  $u^\tr\EE$ has a weak structure of smooth traces.
\end{proposition}
\begin{proof}
	If $f\colon Y\to X$ is a finite flap surjective of degree $d$ map in $\Sm_k$, its transpose ${^tf}$ is a finite correspondence $X\to Y$ such that $f\circ {^tf}=d\cdot\id_X$ in $\Cor(\Sm_k)$. It therefore induces, for any $F\in\s\Pre^\tr(\Sm_k,R)$, a map $\Tr_f\colon f_\ast F_Y\to F_X$ such that $\Tr_f\circ F_f=d\cdot\id_{F_X}$. Since $a\colon X\to\Spec k$ is smooth the functor $\L a^\ast\colon \H^\tr(\Sm_k,R)\to\H^\tr(\Sm_X,R)$ preserves $\A^1$- and Nisnevich-local objects and commutes with $\Omega^{2,1}$ (see \S\ref{sub:basechange}), so we obtain trace maps for the section $(a\colon X\to\Spec k)\mapsto\L a^\ast \EE$ of the $2$-functor $X\mapsto\DM(\Sm_X,R)$ on $\Sm_k$. Since $u^\tr$ is a morphism of $2$-functors which moreover commutes with $\L a^\ast$ for $a$ smooth, there is an induced weak structure of smooth traces on $u^\tr\EE$.
\end{proof}

\begin{corollary}[{\cite[Corollary 5.2.4]{KellyThesis}}]
\label{corollary:MZtraces}
Let $1/c\in R\subset\Q$. Then the motivic cohomology spectrum $\MM R$ has a structures of traces in $\SH(\Sm_k)$.
\end{corollary}

\begin{proof}
By \cite[Theorem 6.4.2]{Levine:2008} (or \cite[Theorem 7.10]{Panin:2012}), the $q$th slice of $\KGL$ is $\Sigma^{2q,q}\MM\Z$. As $s_q$ preserves homotopy colimits, the $q$th slice of $\KGL\tens R$ is $\Sigma^{2q,q}\MM R$, which has a weak structure of smooth traces by Proposition~\ref{prop:MZweakTraces}. Since $\KGL\tens R$ itself has a structure of traces by Proposition~\ref{prop:KGLtraces}, we may apply Theorem~\ref{thm:slicetraces} to deduce that $s_0(\KGL\tens R)\simeq \MM R$ has a structure of traces.
\end{proof}

\begin{proposition}
\label{proposition:moduletraces}
Let $\EE\in\SH(\Sm_k)$ be a motivic spectrum with a structure of traces (\resp{} a weak structure of smooth traces). Then for any $\mathsf F\in\SH(\Sm_k)$, $\EE\wedge \mathsf F$ has a structure of traces (\resp{} a weak structure of smooth traces).
\end{proposition}

\begin{proof}
	By \cite[Theorems 2.3.40 and 1.7.17]{Ayo07}, if $f\colon Y\to X$ is a projective morphism in $\Sch_k$, $\EE\in\SH(\Sm_Y)$, and $\mathsf F\in\SH(\Sm_X)$, the canonical map
	\begin{equation}\label{eqn:exchange}
		f_\ast\EE\wedge\mathsf F \to f_\ast(\EE\wedge f^\ast\mathsf F)
	\end{equation}
is an isomorphism. Thus, if $f\colon Y\to X$ is a finite flat surjective map, we can define $\Tr_f$ as the composition
\[f_\ast (\EE\wedge \mathsf F)_Y\cong f_\ast (\EE_Y\wedge f^\ast\mathsf F_X)\cong f_\ast\EE_Y\wedge\mathsf F_X \xrightarrow{ \Tr_f\wedge \mathsf F_X} \EE_X\wedge \mathsf F_X\cong (\EE\wedge \mathsf F)_X,\]
where the second isomorphism is an instance of~\eqref{eqn:exchange}. It is then easy to verify that this definition satisfies the axioms for a structure of traces (\resp{} a weak structure of smooth traces).
\end{proof}

\begin{corollary}\label{cor:MZmoduleTraces}
	Let $1/c\in R\subset\Q$ and $\EE\in\SH(\Sm_k)$. Then $\MM R\wedge \EE$ has a structure of traces.
	\end{corollary}

\begin{proof}
	Combine Corollary~\ref{corollary:MZtraces} and Proposition~\ref{proposition:moduletraces}.
\end{proof}

We now turn to the proof that $\MM\ZZ_{(\ell)}$-modules are $\ldh$-local.

\begin{definition}\label{def:localSH}
	Let $\tau$ be a Grothendieck topology on $\Sm_k$. A motivic spectrum $\EE\in\SH(\Sm_k)$ is \emph{$\tau$-local} if it is local with respect to the class of maps
	\[\{\Sigma^{p,q}\Sigma^\infty\scr X_+\to\Sigma^{p,q}\Sigma^\infty X_+\suchthat \scr X\to X\text{ is a $\tau$-hypercover and }p,q\in\Z\}.\]
\end{definition}

Since $\Spt(\Sm_k)$ is a stable model category, $\SH(\Sm_k)$ is canonically enriched in the stable homotopy category $\Ho(\Spt)$. That is, for every $\EE,\mathsf F\in\SH(\Sm_k)$, we have a derived mapping spectrum $\R\Hom(\EE,\mathsf F)$ such that $\Omega^\infty\R\Hom(\EE,\mathsf F)\simeq \R\Map(\EE,\mathsf F)$. By \cite[Theorem 3.2.15]{CD}, there exists a presheaf of (symmetric) spectra $\underline{\EE}\colon\Sch_k^\op\to\Spt$ such that the composition of $\underline{\EE}$ with the canonical functor $\Spt\to\Ho(\Spt)$ coincides with the functor
\[(a\colon X\to\Spec k)\mapsto\R\Hom(\1,\R a_\ast\L a^\ast\EE).\]
Note that the restriction of $\underline{\EE}$ to $\Sm_k$ is equivalent to $X\mapsto\R\Hom(\Sigma^\infty X_+,\EE)$. We can thus rephrase Definition~\ref{def:localSH} as follows: $\EE\in\SH(\Sm_k)$ is $\tau$-local if and only if, for every $p,q\in\Z$, the presheaf of spectra $\underline{\Sigma^{p,q}\EE}|_{\Sm_k}$ satisfies $\tau$-descent.

\begin{theorem}[{\cite[Theorem 5.3.7]{KellyThesis}}]
\label{theorem:ldescent}
Every $\ZZ_{(\ell)}$-local motivic spectrum $\EE\in\SH(\Sm_k)$ with a structure of traces is $\ldh$-local. 
\end{theorem}

\begin{proof}
We must show that the presheaf of spectra $\underline{\Sigma^{p,q}\EE}|_{\Sm_k}$ satisfies $\ldh$-descent for every $p,q\in\Z$. We will even prove that $\underline{\Sigma^{p,q}\EE}$ satisfies $\ldh$-descent. We may clearly assume that $p=q=0$ since the hypotheses on $\EE$ are bistable.

As the cdh topology has finite cohomological dimension \cite[Theorem 12.5]{SV:2000} and $\underline{\EE}$ is $\cdh$-local \cite[Proposition 3.7]{Cisinski:2013}, Lemma~\ref{lemma:ThePostnikovtowerargument} will show that $\underline{\EE}$ is $\ldh$-local provided that the canonical maps
\[H^s_{\cdh}(X, a_\cdh\pi_{t}\underline{\EE})\to H^s_{\ldh}(X, a_\ldh\pi_{t}\underline{\EE}) \]
are isomorphisms for all $s,t\in\Z$. By Theorems~\ref{thm:cdh=ldh} and~\ref{theorem:cdhlprimeComparison}, this will be the case if
\begin{enumerate}\setcounter{enumi}{-1}
	\item $\pi_t\underline{\EE}$ is a presheaf of $\zll$-modules with traces,
	\item $\pi_t\underline{\EE}(X)\to \pi_t\underline{\EE}(X_\red)$ is an isomorphism for every $X\in\Sch_k$,
	\item $a_\Nis\pi_t\underline{\EE}|_{\Sm_k}$ is unramified.
\end{enumerate}
(0) holds by assumption on $\EE$ (see Remark~\ref{remark:traces}) and (1) is clear because the restriction functor $\SH(\Sm_X)\to \SH(\Sm_{X_\red})$ is an equivalence of categories.
The sheaf $a_\Nis\pi_t\underline{\EE}|_{\Sm_k}$ is strictly $\A^1$-invariant by \cite[Remark 5.1.13]{Morel:2003}, and in particular is unramified (see \cite[Example 2.3]{Morel:2012}).
\end{proof}

\begin{corollary}\label{corollary:ldescent}
	Every $\MM\Z_{(\ell)}$-module spectrum in $\SH(\Sm_k)$ is $\ldh$-local.
\end{corollary}

\begin{proof}
	Let $\EE$ be an $\MM\Z_{(\ell)}$-module spectrum. Then $\EE$ is a retract of $\MM\zll\wedge\EE$ in $\SH(\Sm_k)$, and so it suffices to show that $\MM\zll\wedge\EE$ is $\ldh$-local. This follows from Corollary~\ref{cor:MZmoduleTraces} and Theorem~\ref{theorem:ldescent}.
\end{proof}

\begin{corollary}[{\cite[Corollary 5.3.8]{KellyThesis}}]
\label{corollary:elldescent1}
Let $R$ be a $\Z_{(\ell)}$-algebra.
If $\X \to X$ is a smooth $\ldh$-hypercover of $X\in\Sm_k$, 
then the induced map $\L R^\tr \X_+ \to \L R^\tr X_+$ is an isomorphism in $\H^\tr_{\Nis,\A^1}(\Sm_k,R)$.
\end{corollary}

\begin{proof}
By \cite[Theorem 1.15]{VV:EMspaces} and Voevodsky's Cancellation Theorem \cite{VV:cancellation}, the stabilization functor 
$\Sigma^\infty\colon\H^\tr_{\Nis,\A^1}(\Sm_k,R)\to\DM(\Sm_k,R)$ is fully faithful. Thus, it suffices to show that any smooth $\ldh$-hypercover 
gives rise to an isomorphism in $\DM(\Sm_k,R)$. 
Recall from \S\ref{sub:EML} that the functor $\L R^\tr\Sigma^\infty(\ph)_+$ factors through the functor 
$\MM R\wedge\Sigma^\infty(\ph):\H^\pt_{\Nis,\A^1}(\Sm_k) \to\Ho(\MM R\Mod)$ to the highly structured 
category of modules over $\MM R$, and this 
reduces the problem to showing that the map $\MM R\wedge\Sigma^\infty\X_+ \to\MM R\wedge\Sigma^\infty X_+$ is
an isomorphism for every smooth $\ldh$-hypercover $\X \to X$. This follows from Corollary~\ref{corollary:ldescent}.
\end{proof}

\begin{corollary}\label{corollary:elldescent}
Let $R$ be a $\Z_{(\ell)}$-algebra. 
Then the localization functor $\H^\tr(\Sm_k,R)\to\H^\tr_{\Nis,\A^1}(\Sm_k,R)$ factors through $\H^\tr_\ldh(\Sm_k,R)$.
\end{corollary}

\begin{proof}
This is just a rephrasing of Corollary~\ref{corollary:elldescent1}.
\end{proof}

\subsection{End of the proof}

We now give the proof of Theorem~\ref{theorem:CDcomm}. As in \cite[end of \S1.3]{VV:EMspaces}, we may assume that $\scr C=\Sm_k$ and $\scr D=\Sch_k$.
Since the restriction functors $i^\ast\colon\H^\pt(\Sch_k)\to\H^\pt(\Sm_k)$ and $i^\ast\colon\H^\tr(\Sch_k,R)\to\H^\tr(\Sm_k,R)$ 
preserve $\A^1$-Nisnevich-local equivalences, it suffices to show that for any $F\in\H^\pt(\Sch_k)$ the canonical map
\[\L R^\tr i^\ast F\to i^\ast\L R^\tr F\]
in $\H^\tr(\Sm_k,R)$ becomes an isomorphism in $\H^\tr_{\Nis,\A^1}(\Sm_k,R)$. In the commutative square
\begin{tikzmath}
	\diagram{\L R^\tr i^\ast\L i_! i^\ast & \L R^\tr i^\ast \\
	i^\ast \L R^\tr\L i_! i^\ast & i^\ast \L R^\tr\rlap, \\};
	\arrows (11-) edge (-12) (21-) edge (-22) (11) edge (21) (12) edge (22);
\end{tikzmath}
the upper horizontal arrow and the left vertical arrow are isomorphisms since $\L i_!$ is fully faithful. Thus, it suffices to show that the lower horizontal arrow becomes an isomorphism in $\H^\tr_{\Nis,\A^1}(\Sm_k,R)$. Since $\L i_! i^\ast F\to F$ is an isomorphism when restricted to $\Sm_k$, it is an $\ldh$-local equivalence by Corollary~\ref{corollary:l'resolution}. Thus, $\L R^\tr \L i_! i^\ast F\to \L R^\tr F$ is an $R^\tr W_\ldh$-local equivalence. By Proposition~\ref{prop:ldhTransfers} (3), $i^\ast \L R^\tr \L i_! i^\ast F\to i^\ast\L R^\tr F$ is therefore an $R^\tr W_{\ldh}$-local equivalence in $\H^\tr(\Sm_k,R)$. Finally, by Corollary~\ref{corollary:elldescent}, $R^\tr W_{\ldh}$-local equivalences in $\H^\tr(\Sm_k,R)$ become equivalences in $\H^\tr_{\Nis,\A^1}(\Sm_k,R)$, as was to be shown.

\section{Applications} 
\label{section:applications}

In this section we record some applications of the results of this paper
which were previously only known for fields of characteristic zero.

\subsection{The structure of the motivic Steenrod algebra and its dual}

The goal of this paragraph is to generalize the structure theorems of Voevodsky for the motivic Steenrod algebra over perfect fields to essentially smooth schemes over fields.
Throughout, the base scheme $S$ is essentially smooth over a field, and $\ell\neq\Char S$ is a fixed prime number. We abbreviate $\MM\Z/\ell$ to $\MM$, and we write $\scr A^\aast$ for the motivic Steenrod algebra at $\ell$, which by Theorem~\ref{theorem:main} is the algebra of all bistable operations in mod $\ell$ motivic cohomology of smooth $S$-schemes, and also the algebra of bigraded endomorphisms of the motivic Eilenberg--Mac Lane spectrum $\MM$. Recall that $\scr A^\aast$ is generated by the reduced power operations $P^i$, the Bockstein $\beta$, and the subalgebra $\HH^\aast(S,\Z/\ell)$. As usual, we write
\[B^i=\beta P^i,\]
and if $\ell=2$,
\[\Sq^{2i}=P^i\quad\text{and}\quad\Sq^{2i+1}=B^i.\]
If $\ell=2$, let $\rho$ be the image of $-1\in\G_{\mathfrak{m}}(S)$ in \[\HH^{1,1}(S,\Z/2)=\HH^1_\et(S,\mu_2)\] and let $\tau$ be the nonvanishing element of \[\HH^{0,1}(S,\Z/2)=\mu_2(S)\cong\Hom(\pi_0(S),\Z/2)\] (recall that $\Char S\neq 2$ if $\ell=2$).

\begin{theorem}[The Adem relations]
\label{theorem:adem}
\leavevmode
\begin{enumerate}
	\item Assume $\ell\neq 2$. If $0<a<\ell b$, then
	\[P^aP^b=\sum_{t=0}^{\lfloor a/\ell\rfloor}(-1)^{a+t}\binom{(\ell-1)(b-t)-1}{a-\ell t}P^{a+b-t}P^t.\]
	If $0< a\leq \ell b$, then
	\begin{align*}
		P^aB^b =&\sum_{t=0}^{\lfloor a/\ell\rfloor}(-1)^{a+t}\binom{(\ell-1)(b-t)}{a-\ell t}B^{a+b-t}P^t \\
	+&\sum_{t=0}^{\lfloor(a-1)/\ell\rfloor}(-1)^{a+t-1}\binom{(\ell-1)(b-t)-1}{a-\ell t}P^{a+b-t}B^t.
	\end{align*}
	\item Assume $\ell=2$ and $0<a<2b$. If $a$ and $b$ are even, then
	\[\Sq^{a}\Sq^{b}=\sum_{t=0}^{\lfloor a/2\rfloor}\tau^{t\,\mathrm{mod}\, 2}\binom{b-t-1}{a-2t}\Sq^{a+b-t}\Sq^t.\]
	If $a$ is even and $b$ is odd, then
	\begin{align*}
		\Sq^{a}\Sq^{b}=&\sum_{t=0}^{\lfloor a/2\rfloor}\binom{b-t-1}{a-2t}\Sq^{a+b-t}\Sq^t\\
		+&\sum_{\substack{t=0\\t\text{ odd}}}^{\lfloor a/2\rfloor}\binom{b-t-1}{a-2t}\rho\Sq^{a+b-t-1}\Sq^t.
	\end{align*}
	If $a$ is odd and $b$ is even, then
	\begin{align*}
		\Sq^{a}\Sq^{b}=&\sum_{\substack{t=0\\t\text{ even}}}^{\lfloor a/2\rfloor}\binom{b-t-1}{a-2t}\Sq^{a+b-t}\Sq^t\\
		+& \sum_{\substack{t=0\\t\text{ odd}}}^{\lfloor a/2\rfloor}\binom{b-t-1}{a-2t-1}\rho\Sq^{a+b-t-1}\Sq^t.
	\end{align*}
	If $a$ and $b$ are odd, then
	\[\Sq^{a}\Sq^{b}=\sum_{\substack{t=0\\t\text{ odd}}}^{\lfloor a/2\rfloor}\binom{b-t-1}{a-2t}\Sq^{a+b-t}\Sq^t.\]
\end{enumerate}
\end{theorem}

\begin{proof}
	If $S$ is a perfect field, this was originally proved in \cite[Theorems 10.2 and 10.3]{VV:motivicalgebra}, but with some typos that were corrected in \cite[Théorèmes 4.5.1 et 4.5.2]{Riou}. In general, choose an essentially smooth morphism $f\colon S\to\Spec k$ where $k$ is a perfect field. The theorem follows from the fact that the base change map $f^\ast\colon\scr A^\aast_k\to\scr A^\aast_S$ is an algebra homomorphism which preserves $P^i$ and $\beta$, and, when $\ell=2$, $\tau$ and $\rho$.
\end{proof}

Note that in the case $\ell=2$, the Adem relations for $a$ odd are obtained from the ones for $a$ even by application of $\beta$, using that $\beta$ is a derivation, that $\beta(\tau)=\rho$, and that $\beta(\rho)=0$.
The Adem relations (together with the relation $\beta^2=0$) can be used to express any monomial in the operations $P^i$ and $\beta$ as a linear combination of the basis operations described in Theorem~\ref{theorem:main}. The only additional piece of information needed to obtain a presentation of $\scr A^\aast$ as an algebra is thus the action of $\scr A^\aast$ on $\HH^\aast(S,\Z/l)$.

Let now $\AAA_\star$ be the dual motivic Steenrod algebra over $S$, that is, \[\AAA_\star=\Hom_{\MM_\star}(\AAA^{-\ast,-\ast},\MM_\star).\]
By duality, $\AAA_\star$ has a structure of commutative Hopf algebroid which is described in details in \cite[\S 12]{VV:motivicalgebra} in the case where $S$ is the spectrum of a perfect field. Our goal is now to identify $\AAA_\star$ with the Hopf algebroid of co-operations in mod $\ell$ motivic cohomology, \ie, $\MM_\star\MM$. 
This is achieved in Propositions \ref{proposition:dualA} and~\ref{proposition:kunneth} below.
For the following lemma, recall that $\Ho(\MM\Mod)$ is a closed symmetric monoidal category.

\begin{lemma}\label{lem:motfinite}
	Let $(p_\alpha,q_\alpha)$ be a family of bidegrees with $p_\alpha\geq\ 2q_\alpha\geq 0$ and such that, for every $q\in\Z$, 
	there are only finitely many $\alpha$ with $q_\alpha\leq q$. Let $\EE=\bigvee_\alpha\Sigma^{p_\alpha,q_\alpha}\MM\in\Ho(\MM\Mod)$. 
	Then the canonical map $\EE\to\Hom_\MM(\Hom_\MM(\EE,\MM),\MM)$ is an equivalence of $\MM$-modules, 
	and the pairing $\pi_\star\Hom_\MM(\EE,\MM)\tens_{\MM_\star}\pi_\star\EE\to\MM_\star$ is perfect.
\end{lemma}

\begin{proof}
	The hypothesis on $(p_\alpha,q_\alpha)$ together with Corollary~\ref{cor:vanish} implies that
	\begin{gather*}
		\EE=\bigvee_\alpha\Sigma^{p_\alpha,q_\alpha}\MM\simeq\prod_\alpha\Sigma^{p_\alpha,q_\alpha}\MM,\\
		\Hom_\MM(\EE,\MM)\simeq\prod_\alpha\Sigma^{-p_\alpha,-q_\alpha}\MM\simeq\bigvee_\alpha\Sigma^{-p_\alpha,-q_\alpha}\MM.
	\end{gather*}
	This immediately implies the first statement and, taking homotopy groups, the second.
\end{proof}

\begin{proposition}
	\label{proposition:dualA}
	There is a canonical isomorphism $\AAA_\star\cong\MM_\star\MM$ and the canonical map $\AAA^{-\ast,-\ast}\to\Hom_{\MM_\star}(\AAA_\star,\MM_\star)$ is an isomorphism.
\end{proposition}

\begin{proof}
By Theorem~\ref{theorem:main}, $\AAA^\star\cong\MM^\star\MM$. 
Therefore it suffices to show there is a perfect pairing $\MM^{-\ast,-\ast}\MM\tens_{\MM_\star}\MM_\star\MM\to\MM_\star$. 
This follows from Theorem~\ref{theorem:main} (3) and Lemma~\ref{lem:motfinite}.
\end{proof}

\begin{lemma}\label{lem:kunneth}
	Let $\RR$ be a motivic $E_\infty$-ring spectrum and let $\EE$ and $\mathsf F$ be $\RR$-modules. The canonical map
	\[\pi_\star\EE\tens_{\RR_\star}\pi_\star\mathsf F\to\pi_\star(\EE\wedge_\RR\mathsf F)\]
	is an isomorphism under either of the following conditions:
	\begin{enumerate}
		\item $\mathsf F$ is a cellular $\RR$-module and $\pi_\star\EE$ is flat over $\RR_\star$.
		\item $\mathsf F\simeq\bigvee_\alpha\Sigma^{p_\alpha,q_\alpha}\RR$ as an $\RR$-module.
	\end{enumerate}
\end{lemma}

\begin{proof}
	Assuming (1), this is a natural transformation between homological functors of $\mathsf F$ that preserve sums, so we may assume (2), in which case the result is obvious.
\end{proof}

\begin{proposition}\label{proposition:kunneth}
Let $\EE$ be an $\MM$-module. Then there is a canonical isomorphism
$$\pi_{\star}\EE\otimes_{\MM_{\star}}\AAA_{\star}\cong\pi_{\star}(\EE\wedge\MM).$$
In particular, $\AAA_\star^{\tens i}\cong\MM_\star(\MM^{\wedge i})$ for all $i\geq 0$ (where the tensor product is over $\MM_\star$).
\end{proposition}
\begin{proof}
	Since $\MM\wedge\MM\simeq\bigvee_\alpha\Sigma^{p_\alpha,q_\alpha}\MM$ and $\MM_\star\MM\cong\AAA_\star$, this is a special case of Lemma~\ref{lem:kunneth}.
\end{proof}

Proposition \ref{proposition:kunneth} is of computational interest because it implies that for any $\EE\in\SH(\Sm_S)$, $\MM_\star\EE$ is a left comodule over $\scr A_\star$. For instance, the $E_{2}$-page of the homological motivic 
Adams spectral sequence is comprised of Ext-groups in the category of $\AAA_{\star}$-comodules. 
See \eg, \cite{DI:2010} and \cite{HKO:2011} for precise statements concerning the construction and convergence of 
the motivic Adams spectral sequence, which can be generalized to essentially smooth schemes over fields by the above results.

We can now give a complete description of the Hopf algebroid $\scr A_\star=\MM_\star\MM$. First, define a Hopf algebroid $(A, \Gamma)$ as follows. Let 
\begin{align*}
	A&=\Z/\ell[\rho,\tau],\\
	\Gamma &=A[\tau_0,\tau_1,\dotsc,\xi_1,\xi_2,\dotsc]/(\tau_i^2-\tau\xi_{i+1}-\rho\tau_{i+1}-\rho\tau_0\xi_{i+1}).
\end{align*}
The structure maps $\eta_L$, $\eta_R$, $\epsilon$, and $\Delta$ are given by the formulas
\begin{align*}
	\eta_L\colon A\to\Gamma,\quad &\eta_L(\rho)=\rho,
	\eta_L(\tau)=\tau,\\
	\eta_R\colon A\to\Gamma,\quad &\eta_R(\rho)=\rho,
	\eta_R(\tau)=\tau+\rho\tau_0,\\
	\epsilon\colon\Gamma\to A,\quad &\epsilon(\rho)=\rho,\\
	&\epsilon(\tau)=\tau,\\
	&\epsilon(\tau_r)=0,\\
	&\epsilon(\xi_{r})=0,\\
	\Delta\colon \Gamma\to\Gamma\tens_A\Gamma,\quad& \Delta(\rho)=\rho\tens 1,\\
	&\Delta(\tau)=\tau\tens 1,\\
	&\Delta(\tau_r)=\tau_r\tens 1+1\tens\tau_r+\sum_{i=0}^{r-1}\xi_{r-i}^{l^i}\tens \tau_i,\\
	&\Delta(\xi_{r})=\xi_r\tens 1+1\tens \xi_r+\sum_{i=1}^{r-1}\xi_{r-i}^{l^i}\tens\xi_i.
\end{align*}
The coinverse map $c\colon\Gamma\to\Gamma$ is determined by the identities it must satisfy. Namely, we have
\begin{align*}
 	&c(\rho)=\rho,\\
 	&c(\tau)=\tau+\rho\tau_0,\\
 	&c(\tau_r)=-\tau_r-\sum_{i=0}^{r-1}\xi_{r-i}^{l^i} c(\tau_i),\\
 	&c(\xi_r)=-\xi_r-\sum_{i=1}^{r-1}\xi_{r-i}^{l^i} c(\xi_i).
\end{align*}

We view $\MM_\aast$ as an $A$-algebra via the map $A\to \MM_\aast$ defined as follows: if $\ell$ is odd it sends both $\rho$ and $\tau$ to $0$, while if $\ell=2$ it sends $\rho$ and $\tau$ to the elements of the same name in $\MM_\aast$.

\begin{theorem}\label{thm:A**}
	$\scr A_\aast$ is isomorphic to $\Gamma\tens_A\MM_\aast$ with
	\[\abs{\tau_r}=(2\ell^r-1,\ell^r-1)\quad\text{and}\quad\abs{\xi_r}=(2\ell^r-2,\ell^r-1).\]
	The map $\MM_\aast\to\scr A_\aast$ dual to the left action of $\scr A^\aast$ on $\MM^\aast$ is a left coaction of $(A,\Gamma)$ on the ring $\MM_\aast$, and the Hopf algebroid $(\MM_\aast,\scr A_\aast)$ is isomorphic to the twisted tensor product of $(A,\Gamma)$ with $\MM_\aast$, \ie,
\begin{itemize}
	\item $\eta_L$ and $\epsilon$ are extended from $(A,\Gamma)$,
	\item $\eta_R\colon H_\aast\to\scr A_\aast$ is the coaction,
	\item $\Delta\colon\scr A_\aast\to \scr A_\aast\tens_{\MM_\aast}\scr A_\aast$ is induced by the diagonal of $\Gamma$ and $\eta_R$ to the second factor,
	\item $c\colon \scr A_\aast\to\scr A_\aast$ is induced by the coinverse of $\Gamma$ and $\eta_R$.
\end{itemize}
\end{theorem}

\begin{proof}
	If $S$ is the spectrum of a perfect field, this is proved in \cite[\S12]{VV:motivicalgebra}. In general, choose an essentially smooth morphism $f\colon S\to\Spec k$ where $k$ is a perfect field. Note that the induced map $(\MM_k)_\aast\to (\MM_S)_\aast$ is a map of $A$-algebras. It remains to observe that the Hopf algebroid $\scr A_\aast$ is obtained from $(\MM_k)_\aast \MM_k$ by extending scalars from $(\MM_k)_\aast$ to $(\MM_S)_\aast$, which follows formally from the following facts: $\L f^\ast$ is a symmetric monoidal functor, Theorem~\ref{thm:EMLpb}, and Corollary~\ref{cor:split}.
\end{proof}

\subsection{Modules over motivic cohomology}
\label{subsection:HZmod}

In this paragraph we generalize the the main result in \cite{RO1}, \cite{RO2}. Recall that there is a symmetric monoidal Quillen adjunction
\[
\Phi\colon{\MM R}\Mod \rightleftarrows \Spt^\tr(\Sm_S,R) :\Psi
\]
which is a Quillen equivalence when $S$ is a field of characteristic zero \cite[Theorem 1]{RO2} or when $S$ is any perfect field and $\Q\subset R$ \cite[Theorem 68]{RO2}. In the following, all functors are derived by default.

\begin{lemma}
\label{lemma:retract}
Let $k$ be a perfect field, $g\colon V\to U$ an $\fpsl$-cover in $\Sm_k$, and $R$ a $\Z_{(\ell)}$-algebra. Then $\MM R\wedge g_+$ has a section in the homotopy category of $\MM R$-modules.
\end{lemma}
\begin{proof}
Since $\Ho(\MM R\Mod)$ is a triangulated category, it suffices to show that $\MM R\wedge g_+$ is an epimorphism, \ie, that for every $\MM R$-module $\EE$, 
\begin{equation}
	\label{eqn:injective}
	[U_+, \EE]\to[V_+, \EE]
\end{equation}
is injective. If $a\colon U\to\Spec k$ is the structure map, we have $U_+=a_\sharp a^\ast\1$ and $V_+=(ag)_\sharp(ag)^\ast\1$. Now for any smooth map $f$, $f_\sharp f^\ast$ is left adjoint to $f_\ast f^\ast$. Using this adjunction, we can identify~\eqref{eqn:injective} with the map
\[
[\1, a_\ast a^\ast \EE]\to [\1, (ag)_\ast (ag)^\ast \EE]
\]
induced by the unit of the adjunction $(g^\ast,g_\ast)$. Since $\EE$ is a retract of $\MM R\wedge \EE$ in $\SH(\Sm_k)$, we can assume that $\EE$ is a free $\MM R$-module. By Propositions~\ref{prop:MZweakTraces} and~\ref{proposition:moduletraces}, $\EE$ then has a weak structure of smooth traces. Since $R$ is a $\ZZ_{(\ell)}$-algebra, it follows that $a_\ast a^\ast \EE\to (ag)_\ast (ag)^\ast \EE$ has a retraction, namely $\frac 1d a_\ast(\Tr_g)$ where $d$ is the degree of $g$.
\end{proof}

\begin{theorem}
\label{theorem:MAlmodules}
Assume that $S$ is the spectrum of a perfect field of characteristic $p>0$ and let $R$ be a commutative ring in which $p$ is invertible. Then $(\Phi,\Psi)$ is a Quillen equivalence.
\end{theorem} 
\begin{proof}
As in \cite{RO2}, it suffices to prove that the unit
$\eta_{X_+}\colon\MM R\wedge X_+\to \Psi\Phi(\MM R\wedge X_+)$
is a weak equivalence for all quasi-projective and connected $X\in \Sm_S$, knowing that this holds if $X$ is projective (more generally, $\eta_E$ is an equivalence if $E\in\SH(\Sm_S)$ is dualizable). A map of $\MM\Z$-modules is an equivalence if and only it induces equivalences of $\MM\Z_{(\ell)}$-modules for every prime $\ell$, and since both $\Phi$ and $\Psi$ preserve filtered homotopy colimits we can assume that $R$ is a $\Z_{(\ell)}$-algebra for some prime $\ell\neq\Char k$. In this case we follow the proof of the rational result \cite[Theorem 68]{RO2} but use Gabber's theorem instead of de Jong's theorem.

We proceed by induction on the dimension of $X$, the case $d=0$ being taken care of since $X$ is then necessarily projective. Let $\SH_d(k)$ be the localizing subcategory of $\SH(\Sm_k)$ generated by shifted suspension spectra of smooth connected $k$-schemes of dimension $\leq d$. By induction hypothesis, $\eta_E$ is an equivalence for every $E\in\SH_{d-1}(k)$. Choose an open embedding $j\colon X\hookrightarrow Y$ into an integral projective $k$-scheme. Let $f\colon Y'\to Y$ be the map given by Gabber's Theorem~\ref{theorem:gabberGlobal}, so that $X'=f^{-1}(X)$ is the complement of a divisor with strict normal crossings. Let $U\subset X$ be an open subset on which $f$ restricts to an $\fpsl$-cover $g\colon V=f^{-1}(U)\to U$. Since $Y'$ is smooth and projective, $Y'_+$ is dualizable in $\SH(\Sm_k)$. By homotopy purity and induction on the number of irreducible components of $Y'\smallsetminus X'$ (\cf{} the proof of \cite[Theorem~52]{RO2}), $X'_+$ is dualizable in $\SH(\Sm_k)$ and hence $\eta_{X'_+}$ is an equivalence. Consider the cofiber sequences
\begin{gather*}
	V_+\hookrightarrow X'_+\to X'/V,\\
	U_+\hookrightarrow X_+\to X/U.
\end{gather*}
 By \cite[Lemma 66]{RO2}, $X'/V$ and $X/U$ belong to $\SH_{d-1}(k)$. By induction hypothesis, $\eta_{X'/V}$ and $\eta_{X/U}$ are equivalences. Thus, $\eta_{V_+}$ is also an equivalence, and it remains to prove that $\eta_{U_+}$ is an equivalence. This follows from Lemma~\ref{lemma:retract}, since $\MM R\wedge U_+$ is a retract of $\MM R\wedge V_+$ in $\Ho(\MM R\Mod)$.
\end{proof}

\vspace{0.5in}
\providecommand{\bysame}{\leavevmode\hbox to3em{\hrulefill}\thinspace}

\vspace{0.1in}

\vspace{0.1in}

\begin{center}
Department of Mathematics, Northwestern University, USA.\\
e-mail: hoyois@math.northwestern.edu
\end{center}

\begin{center}
Fakultät Mathematik, Universität Duisburg-Essen, Germany.\\
e-mail: shanekelly64@gmail.com
\end{center}

\begin{center}
Department of Mathematics, University of Oslo, Norway.\\
e-mail: paularne@math.uio.no
\end{center}
\end{document}